\documentclass[11pt,twoside]{amsart}
\usepackage[latin1]{inputenc}
\usepackage[OT1]{fontenc}
\usepackage{amsmath}
\usepackage{amsthm}
\usepackage{amssymb}
\usepackage[all]{xy}
\usepackage{bbm}
\usepackage{amscd}
\usepackage{a4wide}
\usepackage{enumitem}
\usepackage{mathtools}

\setenumerate[0]{label=(\roman*)}

\DeclareMathOperator{\cc}{\mathsf{c}}

\DeclareMathOperator{\Spec}{\mathsf{Spec}}

\DeclareMathOperator{\pr}{\mathsf{pr}}

\DeclareMathOperator{\id}{\mathsf{id}}
\DeclareMathOperator{\sgn}{\mathsf{sgn}}
\DeclareMathOperator{\Hom}{\mathsf{Hom}}

\DeclareMathOperator{\sHom}{\mathcal{H}\textit{om}}
\DeclareMathOperator{\Ext}{\mathsf{Ext}}

\DeclareMathOperator{\ext}{\mathsf{ext}}
\DeclareMathOperator{\Coh}{\mathsf{Coh}}

\DeclareMathOperator{\sEnd}{\mathcal{E}\textit{nd}}
\DeclareMathOperator{\Ho}{\mathsf H}

\DeclareMathOperator{\rank}{\mathsf{rank}}

\DeclareMathOperator{\triv}{\mathsf{triv}}

\DeclareMathOperator{\im}{\mathsf{im}}

\DeclareMathOperator{\K}{\mathsf K}

\DeclareMathOperator{\tr}{tr}

\DeclareMathOperator{\Ind}{\mathsf{Ind}}

\let\deg\relax
\DeclareMathOperator{\deg}{\mathsf{deg}}

\let\coker\relax
\DeclareMathOperator{\coker}{\mathsf{coker}}

\let\det\relax
\DeclareMathOperator{\det}{\mathsf{det}}

\newcommand{\D}{{\mathsf D}}

\DeclareMathOperator{\Aut}{\mathsf{Aut}}

\DeclareMathOperator{\Pic}{\mathsf{Pic}}

\newcommand{\cU}{{\mathcal U}}
\newcommand{\cW}{{\mathcal W}}

\newcommand{\IC}{\mathbb{C}}
\newcommand{\IE}{\mathbb{E}}

\newcommand{\IN}{\mathbb{N}}
\newcommand{\IP}{\mathbb{P}}

\newcommand{\IZ}{\mathbb{Z}}

\DeclareMathOperator{\Res}{\mathsf{Res}}
\DeclareMathOperator{\VB}{\mathsf{VB}}
\DeclareMathOperator{\FM}{\mathsf{FM}}
\DeclareMathOperator{\TT}{\mathsf{T}\!}

\DeclareMathOperator{\CC}{\mathsf{C}}

\DeclareMathOperator{\WW}{\mathsf{W}}

\DeclareMathOperator{\ho}{\mathsf{h}}
\DeclareMathOperator{\knum}{\mathsf{k}}

\makeatletter
\newcommand{\leqnomode}{\tagsleft@true}
\newcommand{\reqnomode}{\tagsleft@false}
\makeatother

\let\ker\relax
\DeclareMathOperator{\ker}{\mathsf{ker}}

\let\dim\relax
\DeclareMathOperator{\dim}{\mathsf{dim}}

\let\gcd\relax
\DeclareMathOperator{\gcd}{\mathsf{gcd}}

\let\hom\relax
\DeclareMathOperator{\hom}{\mathsf{hom}}

\newcommand{\sym}{\mathfrak S}

\newcommand{\cM}{\mathcal M}
\newcommand{\cN}{\mathcal N}

\newcommand{\alt}{\mathfrak a}

\newcommand{\reg}{\mathcal O}
\newcommand{\I}{\mathcal I}

\newcommand{\eps}{\varepsilon}

\renewcommand{\theta}{\vartheta}
\renewcommand{\rho}{\varrho}
\renewcommand{\phi}{\varphi}
\renewcommand{\_}{\underline{\,\,\,\,}}

\usepackage[usenames,dvipsnames]{xcolor}
\usepackage{hyperref}
\hypersetup{
    colorlinks,
    linkcolor={red!50!black},
    citecolor={blue!50!black},
    urlcolor={blue!80!black}
}
\usepackage{aliascnt} 

\newtheorem{theorem}{Theorem}[section]

  \newaliascnt{proposition}{theorem}
  \newtheorem{prop}[proposition]{Proposition}
  \aliascntresetthe{proposition}

  \newaliascnt{lemma}{theorem}
  \newtheorem{lemma}[lemma]{Lemma}
  \aliascntresetthe{lemma}

  \newaliascnt{corollary}{theorem}
  \newtheorem{cor}[corollary]{Corollary}
  \aliascntresetthe{corollary}

\theoremstyle{definition}

  \newaliascnt{definition}{theorem}
  
  \aliascntresetthe{definition}

  \newaliascnt{remark}{theorem}
  \newtheorem{remark}[remark]{Remark}
  \aliascntresetthe{remark}

  \newaliascnt{condition}{theorem}
  
  \aliascntresetthe{condition}

  \newaliascnt{question}{theorem}
  
  \aliascntresetthe{question}

  \newaliascnt{example}{theorem}
  
  \aliascntresetthe{example}

\begin{document}

\title{Extension groups of Tautological Bundles on Symmetric Products of Curves}
\author[A.\ Krug]{Andreas Krug}

\begin{abstract}
 We provide a spectral sequence computing the extension groups of tautological bundles on symmetric products of curves. One main consequence is that, if $E\neq \reg_X$ is simple, then the natural map $\Ext^1(E,E)\to \Ext^1(E^{[n]},E^{[n]})$ is injective for every $n$. Along with previous results, this implies that $E\mapsto E^{[n]}$ defines an embedding of the moduli space of stable bundles of slope $\mu\notin[-1,n-1]$ on the curve $X$ into the moduli space of stable bundles on the symmetric product $X^{(n)}$. The image of this embedding is, in most cases, contained in the singular locus. For line bundles on a non-hyperelliptic curve, the embedding identifies the Brill--Noether loci of $X$ with the loci in the moduli space of stable bundles on $X^{(n)}$ where the dimension of the tangent space jumps. We also prove that $E^{[n]}$ is simple if $E$ is simple. 
\end{abstract}

\maketitle

\section{Introduction}

Various aspects of tautological bundles on symmetric products of curves have been studied since the early 1960's; see \cite{Schwarzenberger--bundlesplane, Schwarzenberger--secant, Mattuck--sym}. A bit more recently, focuses of research became stability of these bundles \cite{AnconaOttaviani--stab, BohnhorstSpindler--stab, EMLN-stab, BiswasNagaraj-stab, DanPal-stab, BasuDan-stab, Mistretta--stab, Krug--stab} and connections to Koszul cohomology \cite{Voisin--Green1, Voisin--Green2, Ein-Lazarsfeld--gonalityconj}.

Let $X$ be a smooth projective curve over $\IC$ and $n\in \IN$ a positive integer. The symmetric product is the quotient $X^{(n)}=X^n/\sym_n$ by the natural $\sym_n$-action permuting the factors of $X^n$. Let $\pi\colon X^n\to X^{(n)}$ be the quotient morphism and write $x_1+\dots+x_n:=\pi(x_1,\dots, x_n)$ which can be regarded as an effective degree $n$ divisor on $X$. Then, for a vector bundle $E$ on $X$, there is the associated tautological bundle $E^{[n]}$ of $\rank(E^{[n]})=n\rank(E)$, whose fibres are given by
$E^{[n]}(x_1+\dots+x_n)=\Ho^0(E_{\mid x_1+\dots+x_n})$; see \autoref{subsect:taut} for the precise definition.

In this paper, we study the graded vector spaces $\Ext^*(E^{[n]},F^{[n]})$ between two tautological bundles. The technical main results are in \autoref{sect:technicalmain}. There, we describe a spectral sequence, whose terms on the $\IE^1$-level are products of extension and cohomology groups on the curve $X$, and whose limit terms are the extension groups $\Ext^i(E^{[n]},F^{[n]})$.
This spectral sequence follows from computations regarding the $\sym_n$-equivariant pull-back and push-forward along the quotient map $\pi\colon X^n\to X^{(n)}$ carried out in \autoref{sect:push-pull}. These computations are analogous to the ones in the surface case, see \cite{Sca1,Sca2, KruExt, Krug--remarksMcKay}, where the derived McKay correspondence $\D(S^{[n]})\cong \D_{\sym_n}(S^n)$ of \cite{BKR} and \cite{Hai} is used. However, the result for $\Ext^*(E^{[n]},F^{[n]})$ is more complicated in the curve than in the surface case. The reason behind this is the fact that, while the derived McKay correspondence gives equivalences in the surface case, the functors $\pi_*^{\sym_n}\colon \D_{\sym_n}(X^n)\rightleftharpoons\D(X^{(n)}): \pi^*$ in the curve case are not equivalences. For a few more words on the comparison between the curve and the surface case, see \autoref{rem:curvesurf}.

Instead of giving the details of the technical results of \autoref{sect:technicalmain} in this introduction, let us focus on the consequences. In \cite{Krug--stab}, it is shown that, if $E$ is a stable vector bundle with slope $\mu(E)\notin[-1,n-1]$, then the associated tautological bundle $E^{[n]}$ is again $\mu$-stable with respect to the polarisation of $X^n$ given by $H=\reg(x+X^{(n-1)})$ for some $x\in X$. On the critical interval $[-1,n-1]$ the picture is more complicated: There are some stable bundles $E$ with slope in this interval, such that $E^{[n]}$ is not stable any more \cite[Sect.\ 3]{Krug--stab}, but others where stability is still preserved \cite{BiswasNagaraj-stab}. A first application of our spectral sequence for the extension groups is that, at least, simpleness is also preserved on the critical interval.    

\begin{theorem}[\autoref{thm:simpleg2} \& \autoref{thm:simpleg1}]
Let $X$ be a smooth projective curve of genus $g\ge 1$, and $E\neq \reg_X$ a non-trivial vector bundle on $X$. Then, for every $n\in \IN$, we have
\[
 E \text{ is simple} \quad\implies\quad E^{[n]} \text{ is simple.} 
\]
\end{theorem}

Let us fix a pair $(r,d)$ with $d\notin[-r,(n-1)r]$. Then, by the result of \cite{Krug--stab} mentioned above, there is a well-defined map $\phi \colon \cM_X(d,r)\to \cM_{X^{(n)}}$, $[E]\mapsto [E^{[n]}]$ from the moduli space of $H$-slope stable bundles of degree $d$ and rank $r$ to the moduli space of stable bundles on $X^{(n)}$; see \autoref{subsect:inducedFM} for details.  
This map is injective as the isomorphism type of a stable bundle can be reconstructed from the associated tautological bundle; see \cite{Biswas-Nagaraj--reconstructioncurves} for $g\ge 2$ and \cite{Krug--reconstruction} for $g\le 1$. We will prove that the natural map $(\_)^{[n]}\colon \Ext^1(E,E)\to \Ext^1(E^{[n]},E^{[n]})$ is injective for every stable bundle $E\neq \reg_X$; see \autoref{prop:Extinjg2}, \autoref{prop:Extinjg1}. This means that the differential of $\phi$ is injective (see \autoref{prop:FMstab} for details), which implies 

\begin{theorem}\label{thm:mainclosed}
Let $X$ be a smooth projective curve, $d\in \IZ$, $n,r\in \IN$ with $d\notin[-r,(n-1)r]$ and $\gcd(d,r)=1$. 
Let $\cM_X(d,r)$ be the moduli space of stable bundles of rank $r$ and degree $d$ on $X$, and let $\cM_{X^{[n]}}$ be the moduli space of $H$-slope stable bundles on $X^{(n)}$. 
Then the morphism
\[
 \phi\colon \cM_X(d,r)\to \cM_{X^{(n)}}\quad,\quad [E]\mapsto [E^{[n]}]
\]
is a closed embedding.
\end{theorem}

We move on to study properties of the image of the morphism $\phi\colon \cM_X(d,r)\hookrightarrow \cM_{X^{(n)}}$. 
In \autoref{Thm:Ext1} and \autoref{Thm:Ext1d<0}, we give closed formulae for $\Ext^1(E^{[n]},E^{[n]})$ if $E\neq \reg_X$ is stable and $g\ge 2$. For simplicity, here in the introduction, we only state a special case where the direct summands involving Koszul cohomology are guaranteed to vanish.

\begin{theorem}
Let $X$ be non-hyperelliptic of genus $g\ge 3$. Then, for every line bundle $L$ on $X$ of degree $\deg(L)\notin\{0,1\}$, we have
\begin{equation*}
 \Ext^1(L^{[n]},L^{[n]})\cong \Ext^1(L,L)\oplus \Ho^1(\reg_X)\oplus\bigl(\Ho^0(L)\otimes \Ho^1(L^\vee)  \bigr) \oplus\bigl(\Ho^1(L)\otimes \Ho^0(L^\vee)  \bigr)\,.
\end{equation*}
In particular, for any two line bundles $L_1, L_2\in \Pic(X)$ of the same degree $\deg(L_1)=\deg(L_2)\ge 3$ and every $n\ge 2$, we have 
\begin{equation*}
 \ho^0(L_1)<\ho^0(L_2)\quad \iff \quad \ext^1(L_1^{[n]},L_1^{[n]})< \ext^1(L_2^{[n]},L_2^{[n]})\,.
\end{equation*}
\end{theorem}
The second part of the statement can be interpreted by saying that the point $[L^{[n]}]$ of $\cM_{X^{(n)}}$ becomes more singular  when the line bundle $L$ becomes more special. More precisely, for $n\le  d\le 2(g-1)$, the closed embedding $\phi\colon \cM_X(d,1)=\Pic_d(X)\hookrightarrow \cM_{X^{(n)}}$ pulls back the stratification of $\cM_{X^{(n)}}$ given by the dimension of the tangent spaces of points to the Brill--Noether stratification of $\Pic_d(X)$. 

It turns out that, in most cases, the image of $\phi\colon \cM_X(r,d)\to \cM_{X^{(n)}}$ is entirely contained in the singular locus of $\cM_{X^{(n)}}$.  

\begin{theorem}[\autoref{thm:sing}]
Let $E\in \VB(X)$ be a stable vector bundle on a curve of genus $g\ge 3$. For $|\mu(E)|\gg 0$, the point $[E^{[n]}]$ is a singular point of $\cM_{X^{(n)}}$ for every $n\in \IN$, except if $n=2$, and $X$ is non-hyperelliptic of genus $g=3$.
\end{theorem}

In the special case that $n=2$, $X$ is of genus $g=2$ or non-hyperelliptic of genus $g=3$, and $|d|\gg 0$, however, the image $\phi(\Pic_d(X))$ is entirely contained in the smooth locus of $\cM_{X^{(2)}}$; see \autoref{prop:smooth} and \autoref{rem:smooth}. Hence, in these cases, one might conjecture that the whole connected component of $\cM_{X^{(2)}}$ containing $\phi(\Pic_d(X))$ is smooth.

\subsection*{Acknowledgements}
The authors thanks Manfred Lehn for a short but helpful discussion regarding the proof of \autoref{thm:sing} and Pieter Belmans for comments on a first version of this paper.

\section{Preliminaries}

\subsection{Notation and conventions}

We work over the ground field $\IC$.
Throughout, $X$ will be a smooth projective curve. We often write $\reg$ and $\omega$ for the trivial and the canonical line bundle $\reg_X$ and $\omega_X$.

Given a variety $M$, we write $\VB(M)$ for the category of vector bundles and $\Coh(M)$ for the category of coherent sheaves on $M$. Furthermore, we write $\D(M):=\D^b(\Coh(M))$ for the bounded derived category of coherent sheaves. 

Given a positive integer $n$, we write $[n]:=\{1,2,\dots, n\}$.

\subsection{Graded vector spaces}\label{subsect:graded}

A \textit{graded vector space} is a vector space $V$ together with a finite direct sum decomposition $V=\bigoplus_{i\in \IZ} V^i$ where $V^i$ as called the \textit{degree $i$ part} of $V$.
For $d\in \IZ$, the shift $V[d]$ is the graded vector space whose degree $i$ part is $V^{i+d}$. 

On the $k$-fold tensor product $V^{\otimes k}$, there is the $\sym_k$-action given by 
\[
 (i\,,\, i+1)\cdot (v_1\otimes\dots\otimes v_i\otimes v_{i+1}\otimes \dots\otimes v_k)=(-1)^{\deg v_i+\deg v_{i+1}}v_1\otimes\dots\otimes v_{i+1}\otimes v_i\otimes \dots\otimes v_k\,.
\]
We define the symmetric power $S^kV$ and the wedge power $\wedge^kV$ as the invariants and the anti-invariants under this action. There is a natural isomorphism 
\begin{equation}\label{eq:Swedge}
 S^k(V[1])\cong (\wedge^k V)[k]\,.
\end{equation}
The dual $V^\vee$ is the graded vector space whose degree $i$ part is $(V^{-i})^\vee$. 

For $E,F\in \Coh(M)$ we summarise the Ext-spaces in the graded vector space
\[
 \Ext^*(E,F):=\bigoplus_{i\ge 0}\Ext^i(E,F)[-i]\,.
\]
With this notation, Serre duality on a smooth projective variety $Y$ of dimension $n$ takes the form
\[
 \Ext^*(E,F)^\vee[-n]\cong \Ext^*(F,E\otimes \omega_Y)\,. 
\]

\subsection{Tautological bundles}\label{subsect:taut}

Let $X$ be a smooth curve, and $n\in \IN$ a positive integer. The symmetric group $\sym_n$ acts on $X^n$ by permutation of the factors. We denote the quotient by $X^{(n)}=X^n/\sym_n$ and call it the \emph{symmetric product}. By the  Chevalley--Shephard--Todd theorem, $X^{(n)}$ is again smooth. The symmetric product can be identified with the Hilbert scheme of $n$ points on $X$, with the universal family $\Xi=\Xi_n\subset X\times X^{[n]}$ of length $n$ subschemes of $X$ given by the image of the embedding \[X\times X^{(n-1)}\hookrightarrow X\times X^{(n)}\quad,\quad(x,x_1+\dots+x_{n-1})\mapsto (x,x_1+\dots+x_{n-1}+x)\,.\] This means that the Hilbert--Chow morphism 
\[
 \mu\colon X^{[n]}\to X^{(n)}\quad,\quad[\xi]\mapsto \sum_{x\in \xi} \ell(\reg_{\xi,x})x
\]
sending a zero-dimensional length $n$ subscheme to its weighted support is an isomorphism for $X$ a smooth curve.
We denote the projections from the universal family by 
\[X\xleftarrow{\pr_X} \Xi\xrightarrow{\pr_{X^{(n)}}} X^{(n)}\,. 
\]
Since $\pr_{X^{(n)}*}$ is flat and finite of degree $n$, the push-forward $\pr_{X^{(n)}*}$ is exact and sends vector bundles to vector bundles. Hence, for every $E\in \VB(X)$, there is the associated \emph{tautological} vector bundle 
\[
 E^{[n]}:=\pr_{X^{(n)}*}\pr_X^*E \in \VB(X^{(n)}) 
\]
with $\rank(E^{[n]})=n\rank(E)$. 
Since $\Xi$ is also flat over $X$, see \cite[Thm.\ 1.1]{Krug--reconstruction}, the pull-back $\pr_{X}^*$ is exact too. This means that $E\mapsto E^{[n]}$ extends to a functor $(\_)^{[n]}\colon \D(X)\to \D(X^{(n)})$ between the derived categories, and this functor is isomorphic to the Fourier--Mukai transform 
\[
\FM_{\reg_{\Xi}}\colon \D(X)\to\D(X^{(n)})\,. 
\]
Whenever we speak of \emph{stability} of bundles on $X^{(n)}$, we mean slope stability with respect to the polarisation $H=\reg(x+C^{(n-1)})$ for any point $x\in X$. Indeed, the numerical equivalence class of $H$, and hence the notion of slope stability, is independent of the chosen point $x\in C$; see \cite[Sect.\ 1.3]{Krug--stab} for details.

\subsection{Morphisms of moduli spaces induced by Fourier--Mukai transforms}\label{subsect:inducedFM}

Let $X$ and $Y$ be two polarised varieties, and $Z\subset X\times Y$ a closed subscheme which is flat over $X$, and flat and finite over $Y$. We consider the associated Fourier--Mukai transform
\[
 \Phi:=\FM_{\reg_Z}\cong \pr_{Y*}\circ \pr_X^*\colon\D(X)\to \D(Y)
\]
where $\pr_Y\colon Z\to Y$ and $\pr_X\colon Z\to X$ are the projections. 
By the assumptions on $Z$, it restricts to exact functors $\Phi\colon \Coh(X)\to \Coh(Y)$ and $\Phi\colon \VB(X)\to \VB(Y)$. 
Of course, what we have in mind is the case that $X$ is a smooth curve, $Y=X^{(n)}$, and $Z=\Xi_n$ is the universal family of length $n$ subschemes of $X$.  

\begin{prop}\label{prop:FMstab}
Let $\cM$ be a moduli space of stable sheaves on $X$ such that $\Phi(E)$ is a stable sheaf on $Y$ for every $[E]\in \cM$. Then there is a morphism $\phi\colon \cM\to \cN$ to the moduli space $\cN$ of stable sheaves on $Y$ given on closed points by $[E]\mapsto [\Phi(E)]$. Furthermore, for every $[E]\in \cM$ we have a commutative diagram
\begin{equation}\label{eq:CDtang}
\begin{CD}
\TT_\cM([E])
@>{\phi_*}>>
\TT_\cN([\Phi(E)])
\\
@V{\cong}VV
@VV{\cong}V
\\
\Ext^1(E,E)
@>{\Phi}>>
\Ext^1(\Phi(E),\Phi(E))\,.
\end{CD} 
\end{equation}   
\end{prop}

\begin{proof}
We first consider the case that $\cM$ is a fine moduli space, which means that there is a universal family $\cU\in \Coh(\cM\times X)$. Set $\pr_X':=\id_\cM\times \pr_X$, $\pr_Y':=\id_\cM\times \pr_Y$, and
\[\Phi':=\pr_{Y*}'\circ \pr_X'^*\colon \D(\cM\times X)\to\D(\cM\times Y)\,. \] 
Since $\pr_X'$ is still flat and $\pr_Y'$ is still flat and finite, $\Phi'(\cU)$ is a sheaf which is flat over $Y$. Let $[E]\in \cM$ be a closed point given by $i_{[E]}\colon \Spec\IC\to \cM$. Then, $(i_{[E]}\times\id_X)^* \cU\cong E$. By base change along the diagram 
\[
\begin{CD}
X @<{\pr_X}<<
Z
@>{\pr_Y}>>
Y
\\
@V{i_{[E]}\times \id_X}VV
@V{i_{[E]}\times \id_Z}VV
@V{i_{[E]}\times \id_Y}VV
\\
\cM\times X @<{\pr'_X}<<
\cM\times Z
@>{\pr'_Y}>>
\cM\times Y
\end{CD} 
\]   
we get $(i_{[E]}\times \id_Y)^*\Phi'(\cU)\cong \Phi(E)$ which is stable by assumption. Hence $\Phi'(\cU)$ is a flat family of stable sheaves on $Y$, which means that we get a classifying morphism $\phi\colon \cM\to \cN$ with 
$\phi([E])=[\Phi(E)]$ for every closed point $[E]\in \cM$. 

By \cite[Prop.\ 3.3.2]{Cald-thesis}, there is an \'etale covering $\{U_i\}$ of $\cM$ such that there are universal families $\cU_i$ over $U_i\times X$ together with isomorphisms $\cU_{i\mid (U_i\times_\cM U_k)\times X}\cong \cU_{k\mid (U_i\times_\cM U_k)\times X}$.\footnote{Note that, in general, these isomorphisms only satisfy a \emph{twisted} cocycle condition, in which the local universal families do not glue to give a universal family as an ordinary sheaf, but only as a twisted sheaf.} Hence, we can construct the morphism $\phi$ first locally using the universal family as above, and then glue the local pieces by \'etale descent.    

Since the statement on the differential is an \'etale local one, we can assume that both, $\cM$ and $\cN$, have universal families $\cU$ and $\cW$ with  
\begin{equation}\label{eq:classpull}
 (\phi\times \id_Y)^*\cW\cong \cU\,.
\end{equation}
The vertical isomorphisms in \eqref{eq:CDtang} are given as follows (we explain the left one, but the right one is analogous): A tangent vector $v\in \TT_{\cM}([E])$ corresponds to a morphism $j_v\colon \Spec \IC[\eps]\to \cM$ with $r\circ j_v=i_{[E]}$, where $r\colon \Spec\IC\to \Spec\IC[\eps]$ is the inclusion of the reduction. Let $E_v:=(j_v\times \id_X)^*\cU$. Then, there is a short exact sequence
\begin{equation}\label{eq:sestang}
 0\to E\to E_v\to E\to 0\,, 
\end{equation}
where the second map is the restriction to the reduction. The class $c_v\in \Ext^1(E,E)$ of \eqref{eq:sestang} is the image of $v$ under the vertical isomorphism $\TT_{\cM}([E])\to \Ext^1(E,E)$ in \eqref{eq:CDtang}. 

Now, the tangent vector $\phi_*(v)\in \TT_{\cN}([\Phi(E)])$ corresponds to $j_{\phi_*(v)}=\phi\circ j_v\colon \Spec\IC[\eps]\to \cN$. Let $\Phi(E)_{\phi_*(v)}:=(j_{\phi_*(v)}\times \id_Y)^*\cW$. By base change along the diagram 
\[
\begin{CD}
X @<{\pr_X}<<
Z
@>{\pr_Y}>>
Y
\\
@V{r\times \id_X}VV
@V{r\times \id_Z}VV
@V{r\times \id_Y}VV
\\
\Spec\IC[\eps]\times X @<{\id \times \pr_X}<<
\Spec\IC[\eps]\times Z
@>{\id\times \pr_Y}>>
\Spec\IC[\eps]\times Y  @= \Spec\IC[\eps]\times Y
\\
@V{j_v\times \id_X}VV
@V{j_v\times \id_Z}VV
@V{j_v\times \id_Y}VV
@V{j_{\phi_*(v)}\times \id_Y}VV
\\
\cM\times X @<{\pr'_X}<<
\cM\times Z
@>{\pr'_Y}>>
\cM\times Y
@>{\phi\times \id_Y}>>
\cN\times Y
\end{CD} 
\]   
together with \eqref{eq:classpull}, we see that 
the short exact sequence
\[
 0\to \Phi(E)\to \Phi(E)_{\phi_*(v)}\to \Phi(E)\to 0\,,
\]
where the second map is the restriction to the reduction, is given by applying $\Phi$ to \eqref{eq:sestang}. Hence, $c_{\phi_*(v)}=\Phi(c_v)$.
\end{proof}

\begin{remark}\label{rem:mainsuffice}
We apply \autoref{prop:FMstab} to $\Phi=(\_)^{[n]}=\FM_{\reg_{\Xi_n}}\colon \D(X)\to \D(X^{(n)})$ and $\cM=\cM_X(r,d)$ where $d\notin [-r,(n-1)r]$. Then, by \cite{Krug--stab}, the assumption of \autoref{prop:FMstab} is fulfilled: $E^{[n]}$ is stable for every $[E]\in \cM_X(r,d)$. We get the classifying morphism $\phi\colon \cM_X(r,d)\to\cM_{X^{(n)}}$, $\phi([E])=[E^{[n]}]$.

If $\gcd(r,d)=1$, this morphism is projective. Hence, by \cite[Thm.\ 1.11]{Vakil--note43}, it is a closed embedding if and only if it is injective on closed points and on tangent vectors. 
The injectivity on closed points is the fact that $E^{[n]}\cong F^{[n]}$ implies $E\cong F$; see \cite[Thm.\ 1.1]{Biswas-Nagaraj--reconstructioncurves} for $g\ge 2$, \cite[Thm.\ 1.3]{Krug--reconstruction} for $g=1$, and \cite[Rem.\ 4.4]{Krug--reconstruction} for $g=0$. By \autoref{prop:FMstab}, injectivity on the tangent vectors will follow from the injectivity of $(\_)^{[n]}\colon \Ext^1(E,E)\to \Ext^1(E^{[n]},E^{[n]})$. For $g=0$, we have $\Ext^1(E,E)\cong \Hom(E,E\otimes \omega_{\IP^1})=0$ for every simple bundle $E$. Otherwise, composition with the embedding $E\otimes \IP^1\hookrightarrow E$ would give a non-trivial automorphism of $E$. 
For $g\ge 2$, injectivity of $(\_)^{[n]}$ on $\Ext^1$ is proved in \autoref{prop:Extinjg2}, and for $g=1$, this is  
\autoref{prop:Extinjg1}.
\end{remark}

\subsection{Equivariant sheaves}

In this subsection, we summarize some results on equivariant sheaves. For further details, we refer to \cite[Sect.\ 4]{BKR} or \cite{Elagin}.
Let $G$ be a finite group acting on a variety $M$. A \emph{$G$-equivariant} (coherent) sheaf is a pair $(E,\lambda)$ where $E\in \Coh(E)$ and $\lambda$ is a \emph{$G$-linearisation} of $E$, i.e.\ a family of isomorphism $\{\lambda_g\colon E\xrightarrow \sim g^*E\}_{g\in G}$ satisfying $\lambda_{gh}=(h^*\lambda_g)\circ \lambda_h$ for all $g,h\in G$. Later, we often write just $E$ instead of $(E,\lambda)$ with the linearisation omitted in the notation. Given two equivariant sheaves $(E,\lambda)$ and $(F,\mu)$, there is a $G$-action on $\Hom(E,F)$ by conjugation by the linearisations: $g\cdot \phi=\mu_g^{-1}\circ (g^*\phi)\circ \lambda_g$. We get an abelian category $\Coh_G(X)$ with $G$-equivariant coherent sheaves as objects, and morphism given by
\[
 \Hom_G\bigl((E,\lambda),(F,\mu)\bigr) :=\Hom(E,F)^G
\]
where the invariants on the right side are taken with regard to the above conjugation action. Similarly, there is an induced action by conjugation on $\Ext^i(E,F)$ and we have
\[
 \Ext^i_G\bigl((E,\lambda),(F,\mu)\bigr)\cong \Ext^i(E,F)^G
\]
where the left side is given by the derived functor of $\Hom_G$ and the right side is given by the invariants under the conjugation action. The structure sheaf $\reg_M$ has a canonical linearisation by push-forward of regular functions along the automorphisms of $G$, and we write 
\[
 \Ho^*_G(M,F):=\Hom_G(\reg_M,F)\cong \Ho^*(M,F)^G\,.
\]
The pull-backs and push-forwards of equivariant sheaves along equivariant morphisms inherit linearisations. This means that, for $f\colon M\to N$ a $G$-equivariant morphism of varieties, there is a pull-back $f^*\colon \Coh_G(N)\to \Coh_G(M)$ and, if $f$ is proper, also a push-forward $f_*\colon \Coh_G(M)\to \Coh_H(N)$ with $f^*\dashv f_*$.

Also, the tensor product of equivariant sheaves is canonically equipped with a linearisation. 
In particular, let $E\in \Coh_G(M)$ be an equivariant sheaf, and $\rho$ a $G$-representation. Then, we can define a new equivariant sheaf $E\otimes \rho:=E\otimes p^*\rho$ where $p\colon X\to\Spec\IC$ is the structure morphism. If $\chi\colon G\to \IC$ is a one-dimensional representation, i.e.\ a character, $E\otimes\chi$ has the same underlying sheaf $E$, but the linearisation is twisted by $\chi$. In particular, if $\alt$ denotes the sign representation of $\sym_n$ and $E=(E,\lambda)$, then $E\otimes\alt=(E,\overline \lambda)$ with $\overline \lambda_g=\sgn(g)\lambda_g$.

Let $H\subset G$ be a subgroup. There is the \emph{restriction} functor \[\Res=\Res_G^H\colon \Coh_G(M)\to \Coh_H(G)\quad, \quad (E,\lambda)\mapsto (E,\lambda_{\mid H})\] which restricts the $G$-linearisation of an equivariant sheaf to an $H$-linearisation. There is also the \emph{induction} functor $\Ind=\Ind_H^G\colon \Coh_H(M)\to \Coh_G(M)$ which is left- and right-adjoint to $\Res_G^H$. We have $\Ind_G^H(E,\lambda)=\bigoplus_{[g]\in G/H}g^*E$ with the linearisation given by a combination of the linearisation of $E$ and permutation of the factors. Induction and restriction commute with pull-backs and push-forwards along equivariant morphisms. In particular, for $f\colon M\to N$ a proper $G$-equivariant morphism, we have
\begin{equation}\label{eq:pushInd}
 f_*\circ \Ind_H^G\cong \Ind_H^G\circ f^*\,. 
\end{equation}

If $G$ acts trivially on $M$, a $G$-linearisation of a sheaf $E$ is the same as a $G$-action, i.e. a homomorphism $G\to \Aut(E)$. In this case, there is the functor $\triv\colon \Coh(M)\to \Coh_G(M)$ which equips every sheaf with the trivial linearisation, and the functor $(\_)^{G}\colon \Coh_G(M)\to\Coh(M)$ which takes the invariants under the $G$-action. These two functors are both-sided adjoints of each other.
Let $\pi\colon N\to N/G$ be a quotient morphism (of course, we mainly want to consider the case $\pi\colon X^{n}\to X^{(n)}=\sym_n$). Then, $\pi$ is $G$-equivariant when considering the quotient $N/G$ equipped with the trivial action. We have an adjoint pair $\pi^*\dashv \pi_*^{G}$, where $\pi^*:=\pi^*\circ \triv \colon \Coh(N/G)\to \Coh_{G}(N)$ (so $\triv$ is usually hidden in the notation) and $\pi_*^{G}:=(\_)^{G}\circ \pi_*\colon \Coh_{G}(N)\to \Coh(N/G)$. Since $\pi_*^G \reg_N\cong \reg_{N/G}$, projection formula gives
\begin{equation}\label{eq:projG}
 \pi_*^G\circ \pi^*\cong \id_{\Coh(N/G)}\,.
\end{equation}

Let still $G$ act trivially on $M$, and let $E$ be a $G$-equivariant sheaf. Assume that there is a direct sum decomposition $\bigoplus_{i\in I} E_i$ and a $G$-action on the index set $I$ such that $g\cdot E_i=E_{g\cdot i}$ for all $g\in G$ and $i\in I$. Let $i_1,\dots,i_k$ be a set of representatives of the $G$-orbits of $I$, and let $G_{i_j}$ be the isotropy groups. Then, the projection $E\to E_{i_1}\oplus\dots \oplus E_{i_k}$ induces an isomorphism
\begin{equation}\label{eq:Danila}
 E^G\xrightarrow\cong E_{i_1}^{G_{i_1}}\oplus \dots \oplus E_{i_k}^{G_{i_k}}\,.
\end{equation}
Its inverse $E_{i_1}^{G_{i_1}}\oplus \dots \oplus E_{i_k}^{G_{i_k}}\to E^G$ is given by
\begin{equation}\label{eq:Daninv}
(s_1,\dots,s_k)\mapsto \sum_{g\in G}g\cdot\iota\Bigl(\frac{s_1}{|G_{i_1}|},\dots , \frac{s_k}{|G_{i_k}|}\Bigr)= \sum_{j=1}^k \sum_{[g]\in G/G_{i_j}} g\cdot \iota_j(s_j) 
\end{equation}
where $\iota\colon E_{i_1}\oplus \dots \oplus E_{i_k}\to E$ and $\iota_j\colon E_{i_j}\to E$ are the inclusions of the direct summands. Note that we can apply this to the special case $M=\Spec \IC$ where $G$-equivariant sheaves are the same as $G$-representations. This will be used later to compute equivariant extension groups.

If $G$ acts transitively on $I$, we have $E\cong \Ind_{G_{i}}^G E_i$ for every $i\in I$. Hence, in this special case, \eqref{eq:Danila} gives an isomorphism of functors 
\begin{equation}\label{eq:IndDanila}
 (\_)^G\Ind_H^G\cong (\_)^H\colon \Coh_H(M)\to \Coh(M)
\end{equation}
for every subgroup $H\subset G$.

Given an action of a finite group on a smooth variety $M$, we write $\D_G(M):=\D^b(\Coh_G(M))$ for the bounded derived category of coherent $G$-equivariant sheaves on $G$. All of the functors discussed above can be derived to give functors between the equivariant derived categories. Note, however, that many of the functors, like $\Ind$, $\Res$, $(\_)^G$ and $\triv$ are already exact, so their derived versions are just given by term-wise application to complexes. For $E,F\in \Coh_G(M)$, which we can regard as complexes concentrated in degree zero, we have
\[
 \Ext^i_G(E,F)\cong \Hom_{\D_G(M)}(E,F)\,.
\]
Accordingly, we also write $\Ext^i_G(E^\bullet, F^\bullet):= \Hom_{\D_G(M)}(E^\bullet,F^\bullet)$ if $E^\bullet, F^\bullet \in \D_G(M)$ are proper complexes.

\section{Pull-back and push-forward along the quotient map}\label{sect:push-pull}

Throughout this section, let $n\ge 2$ be some fixed number. For this section, X can be any smooth curve and does not need to be projective. Let $\pi\colon X^n\to X^{(n)}$ be the quotient morphism. For $E\in \VB(X)$, we will see that there is an $\sym_n$-equivariant vector bundle $\CC(E)\in \VB_{\sym_n}(X^n)$ with $\pi_*^{\sym_n}\CC(E)\cong E^{[n]}$. We will also describe the pull-back $\pi^*E^{[n]}\in \VB_{\sym_n}(X^n)$. 

Let $\pr_i\colon X^n\to X$, $(x_1,\dots, x_n)\mapsto x_i$ be the projection to the $i$-th factor.
We consider the functor \[
             \CC:=\Ind_{\sym_{n-1}}^{\sym_n}\circ \pr^*_1\colon \VB(X)\to \VB_{\sym_n}(X^n)\,.
             \]
For this definition to make sense, we note that $\pr_1\colon X^n\to X$ is $\sym_{n-1}$-invariant, if we consider $\sym_{n-1}\cong \sym_{[2,n]}\le \sym_n$ as the subgroup of permutations which leave $1$ fixed. That means that $\CC(E)\cong \bigoplus_{i=1}^n\pr_i^*E$ and the linearisation $\lambda$ of $\CC(E)$ has the property that for $g\in \sym_n$ and $i\in [n]$, we have $\lambda_g(E_i)=g^*E_{g(i)}$. Furthermore, for a morphism $\phi\colon E\to F$ of vector bundles on $X$, we have $\CC(\phi)=\oplus_i \pr_i^*(\phi)\colon \CC(E)\to \CC(F)$. We can also describe the functor $\CC$ as the equivariant Fourier--Mukai transform
\[
 \CC\cong\FM_{\Ind_{\sym_{n-1}}^{\sym_n}\reg_{D_1}}\cong \FM_{\oplus_{i=1}^n \reg_{D_i}}\,,
\]
where $D_i=\{(x;x_1,\dots, x_n)\mid x=x_i\}\subset X\times X^n$.

\begin{prop}\label{prop:push}
There is an isomorphism of functors $(\_)^{[n]}\cong \pi_*^{\sym_n}\circ \CC\colon \D(X)\to \D(X^{(n)})$. 
\end{prop}

\begin{proof}
This should be well-known. For example, the case of line bundles is \cite[Prop.\ 1]{Mattuck--sym}. Anyway, let us do the proof as it is not long.
 We use the commutative diagram
 \[
\begin{xy}
\xymatrix{
& X^n \ar_{\pr_1}[ld]  \ar^{q_1}[d] \ar^\pi[rd] &    \\
 X  & \Xi=X\times X^{(n-1)} \ar^{\pr_X\qquad\quad}[l]\ar_{\quad\qquad\pr_{X^{(n)}}}[r] & X^{(n)} }
\end{xy} 
\]
Since $q_1$ is the $\sym_{n-1}$-quotient morphism, we get
\begin{align*}
(\_)^{[n]}\cong \pr_{X^{(n)}*} \pr_{X}^*\overset{\eqref{eq:projG}}\cong \pr_{X^{(n)}*} q_{1*}^{\sym_{n-1}}q_1^*\pr_{X}^*\cong (\_)^{\sym_{n-1}}\pi_*\pr_1^*&\overset{\eqref{eq:IndDanila}}\cong (\_)^{\sym_{n}}\Ind_{\sym_{n-1}}^{\sym_n}\pi_*\pr_1^* \\&\overset{\eqref{eq:pushInd}}\cong \pi_*^{\sym_{n}}\Ind_{\sym_{n-1}}^{\sym_n}\pr_1^*  \\&\cong \pi_*^{\sym_{n}}\CC\,.\qedhere
\end{align*}
\end{proof}

The functors in \autoref{prop:push} restrict to the categories of vector bundles so that we get a natural isomorphism $\pi_*^{\sym_n}\CC(E)\cong E^{[n]}$ for every $E\in \VB(X)$. 

\begin{lemma}\label{lem:pullFM}
We consider the reduced subscheme $D\subset X\times X^n$ given by 
\[
 D=\bigcup_{i=1}^n D_i=\bigl\{(x;x_1,\dots,x_n\mid x=x_i \text{ for some } i\in[n]\bigr\}\,. 
\]
Then, the following diagram is cartesian:
\begin{align}\label{eq:cdD}
\begin{CD}
D
@>{\pr_{X^n}}>>
X^n
\\
@V{\id_X\times \pi}VV
@VV{\pi}V
\\
\Xi
@>{\pr_{X^{(n)}}}>>
X^{(n)}\,.
\end{CD} 
\end{align}

\end{lemma}
\begin{proof}
Recall that $\Xi\cong X\times X^{(n-1)}$. As a subset of $X^n\subset X\times X^{(n-1)}\times X^n$, the fibre product $F:=\Xi \times_{X^{[n]}}X^n$ is given by
\[
 F=\bigl\{(x, x_2+\dots+x_n, y_1,\dots, y_n)\mid x+x_2+\dots+x_n=y_1+\dots+ y_n  \bigr\}\,.
\]
Since \eqref{eq:cdD} is commutative, we get an induced map $D\to F$. Let us show that this is an isomorphism with inverse given by $\pr_{X\times X^n\mid F}$ where $\pr_{X\times X^n}\colon X\times X^{(n-1)}\times X^n\to X\times X^n$ is the projection. It is easy to check that the composition of these morphisms in both directions equal the identities on $D$ and $F$, respectively, on $\IC$-valued points. If both $D$ and $F$ are reduced, this implies that the compositions equal the identities.

The subscheme $D\subset X\times X^n$ is reduced by definition. 
As $\pi\colon X^n\to X^{(n)}$ is flat and finite, the fibre product $F$ is flat and finite over the smooth variety $\Xi\cong X\times X^{(n-1)}$, hence Cohen--Macaulay; see \cite[Exe.\ 18.17]{Eisenbud--commutativebook}. It follows that $F$ is reduced since it is generically reduced; see \cite[Exe.\ 18.9]{Eisenbud--commutativebook}. 
\end{proof}

\begin{cor}
There is an isomorphism of functors $\pi^*\circ (\_)^{[n]}\cong \FM_{\reg_D}\colon \D(X)\to \D(X^{(n)})$. 
\end{cor}

\begin{proof}
 This follows from base change along the diagram \eqref{eq:cdD}. 
\end{proof}

In order to describe the $\sym_n$-equivariant bundle $\pi^*E^{[n]}\in \VB_{\sym_n}(X^n)$ more concretely, we will define an $\sym_n$-equivariant complex $\CC^\bullet_E$ concentrated in degrees $0,\dots n-1$ as follows.  
We start by setting $\CC^0_E:=\CC(E)$.
For $2\le k\le n$, we consider the embedding 
\[\iota_{[k]}\colon X\times X^{n-k}\hookrightarrow X^n\quad,\quad \iota(x,x_1,\dots, x_{n-k})=(\underbrace{x,\dots, x}_{\text{$k$-times}}, x_1,\dots, x_{n-k})\,.\]
We note that 
$\iota_{[k]*}(E\boxtimes \reg_{X^{n-k}})$ carries a natural $\sym_k\times \sym_{n-k}$-linearisation where we consider $\sym_k\times \sym_{n-k}\cong \sym_{[k]}\times\sym_{[k+1,n]} \le \sym_n$ as the subgroup of permutations which fix the two blocks $[k]$ and $[k+1,n]$. We define
\[
\CC^{k-1}_E:=\Ind_{\sym_k\times \sym_{n-k}}^{\sym_n} \iota_{[k]*}\bigl((E\otimes \alt_k)\boxtimes \reg_{X^{n-k}}\bigr)
\]
where $\alt_k$ is the sign representation of $\sym_k$.
In order to give a more concrete description of $\CC^{k-1}_E$, for $I\subset [n]$ with $|I|=k$, we set $\iota_I=g\circ \iota_{[k]}\colon X\times X^{n-k}$ where $g\in \sym_n$ is a permutation with $g([k])=I$ such that $g_{\mid [n]\setminus I}$ is increasing. The image of $\iota_I$ is the \textit{$I$-th partial diagonal}
$\Delta_I=\bigl\{(x_1,\dots,x_n)\in X^n\mid x_i=x_j\,\forall\, i,j\in I\bigr\}$. We then have 
\begin{equation}\label{eq:EI}
\CC^{k-1}_E=\bigoplus_{I\subset [n]\,,\, |I|=k} E_I\quad,\quad E_I:=\iota_{I*}(E\boxtimes \reg_X)\,, 
\end{equation}
and its linearisation $\lambda$ satisfies $\lambda_g(E_I)=E_{g(I)}$ for $g\in \sym_n$, and $\lambda_{h\mid E_I}$ is multiplication by $\sgn(h)$ for $h\in \sym_I\le \sym_n$.

Finally, we define differentials $d^p\colon \CC^p_E\to \CC^{p+1}_E$ by 
\begin{equation}\label{eq:ddef}
(d^p(s))_I:=\sum_{i\in I}(-1)^{|\{j\in I\mid j<i\}|} s_{I\setminus\{i\}}\,. 
\end{equation}
Here, for a local section $s\in \CC^p_E$, we denote its component in $E_J$ by $s_J$. For $p=0$, we make sense of \eqref{eq:ddef} by setting $E_{\{j\}}:=\pr_j^*E$.

\begin{prop}\label{prop:epsres}
Let $E\in \VB(X)$, and let $\eps:=\eps_{\CC^0_E}\colon \pi_*\pi_*^{\sym_n}\CC^0_E\to \CC_E^0$ be  the counit of the adjunction $\pi^*\dashv \pi_*^{\sym_n}$.  Then, the following $\sym_n$-equivariant sequence is exact: 
\[
0\to \pi^*\pi_*^{\sym_n}\CC^0_E\xrightarrow{\eps} \CC^0_E\xrightarrow{d^0} \CC^1_E\xrightarrow{d^1} \dots\xrightarrow{d^{n-1}} \CC^{n-1}_E\to 0\,. 
\]
In particular, there is an isomorphism $\pi^*E^{[n]}\cong \CC^\bullet_E$ in $\D_{\sym_n}(X^n)$.
\end{prop}

\begin{proof}
For $I\subset [n]$, we set $D_I:=\cap_{i\in I} D_i$.
Since the irreducible components $D_i$ of $D$ intersect transversely, we get an $\sym_n$-equivariant resolution 
\begin{align}\label{eq:ODres}
0\to \reg_D\to \bigoplus_{i=1}^n \reg_{D_i} \to \bigoplus_{|I|=2}\reg_{D_I}\to \dots \to \reg_{D_{[n]}}\to 0 
\end{align}
of $\reg_D$ with the differentials and the $\sym_n$-linearisations analogous to those of $\CC^\bullet_E$. This is exactly as in the surface case, where details can be found in \cite[Rem.\ 2.2.1]{Sca1}.
We have 
\[
\FM_{\bigoplus_{|I|=p+1}\reg_{D_I}}(E)\cong \CC^p_E\,.
\]
Hence, \eqref{eq:ODres} induces a short exact sequence
\begin{equation} \label{eq:adjointresolution}
0\to \FM_{\reg_D}(E)\to  \CC^0_E\to \CC^1_E\to \dots\to \CC^{n-1}_E\to 0\,. 
\end{equation}
This is also exactly as in the surface case; see  \cite[Thm.\ 2.2.3]{Sca1} or \cite[Thm.\ 16]{Sca2} or \cite[Rem.\ 2.10]{Krug--remarksMcKay} for details.

By \autoref{prop:push} and \autoref{lem:pullFM}, we have an isomorphism $\FM_{\reg_D}(E)\cong \pi^*\pi_*^{\sym_n}\CC^0_E$. So it is left to show that, under this isomorphism, the first map in \eqref{eq:adjointresolution} agrees, up to a scalar multiple, with the counit of adjunction $\eps$. Note that both functors, $\pi_*^{\sym_n}$ and $\pi^*$, are Fourier--Mukai transforms. Hence, by \cite[App.\ A]{CW}, the counit $\pi^*\circ \pi_*^{\sym_n}\to \id_{\D_{\sym_n}}(X^n)$ is induced by a map between the Fourier--Mukai kernels. Hence, after precomposition by the functor $\CC=\FM_{\oplus_i\reg_{D_i}}$, the counit $\eps\colon \pi^*\circ \pi_*^{\sym_n}\circ \CC\to  \CC$ is still induced by some $\sym_n$-equivariant map $\reg_D\to \oplus_{i=1}^n\reg_{D_i}$ between the Fourier--Mukai kernels. However, one easily checks that, up to scalar multiples, the only $\sym_n$-equivariant map $\reg_D\to \oplus_{i=1}^n\reg_{D_i}$ is the natural injection given by the restriction map to every component. This is the same as the   
the first map in \eqref{eq:ODres} and hence it induces the first map of \eqref{eq:adjointresolution}.

The quasi-isomorphism $\pi^*E^{[n]}\cong \CC^\bullet_E$ now comes from the isomorphism $\pi^*E^{[n]}\cong \pi^*\pi_*^{\sym_n}\CC^0_E$, see \autoref{prop:push}, together with the fact, which we just proved, that $\CC^\bullet_E$ is an equivariant resolution of $\pi^*\pi_*^{\sym_n}\CC^0_E$.
\end{proof}

\begin{remark}\label{rem:curvesurf}
The results of this subsection are analogous to the surface case. For $S$ a smooth quasi-projective surface, tautological bundles on the Hilbert scheme $S^{[n]}$ of points on the surface are still defined by means of the Fourier--Mukai transform along the universal family. A crucial difference is that, in the surface case, the Hilbert--Chow morphism $\mu\colon S^{[n]}\to S^{(n)}$ from the Hilbert scheme of point to the symmetric product is \textit{not} an isomorphism. However, one can consider the commutative diagram 
\begin{equation}\label{eq:Haiman}
\begin{CD}
I^nS
@>{p}>>
S^n
\\
@V{q}VV
@VV{\pi}V
\\
S^{[n]}
@>{\mu}>>
S^{(n)}\,.
\end{CD} 
\end{equation}
where $I^nS:=(S^n\times_{S^{(n)}}S^{[n]})_{\mathsf{red}}$. Then, instead of the pull-back and push forward along $\pi$, one considers the \textit{derived McKay correspondences} 
\[
 \Phi:=Rp_*\circ q^*\colon \D(S^{[n]})\to \D_{\sym_n}(S^n)\quad,\quad \Psi:= q_*^{\sym_n}\circ Lp^*\colon \D_{\sym_n}(S^n)\to \D(S^{[n]})\,.
\]
For $E\in \VB(S)$, the equivariant bundle $\CC(E)=\CC^0_E$ and the complex $\CC^\bullet_E$ can then be defined in the exact same way as in the curve case above, and we have $\Psi(\CC(E))\cong E^{[n]}$ and $\Phi(E^{[n]})\cong \CC^\bullet_E$; see \cite{Sca1,Sca2, Krug--remarksMcKay}. The proofs we did in this section are very similar to the proofs in \cite{Sca1,Sca2, Krug--remarksMcKay}, with the main difference being that some steps become easier in the curve case because the diagram \eqref{eq:Haiman} collapses in the sense that the left and right side are identified. Instead of reproducing the proofs in the curve case, one can also deduce them from the surface case by a base change argument. However, we needed to go through parts of the argument again, since the statement that the augmentation map in \autoref{prop:epsres} is given by the counit of adjunction does not appear in any of the references on the surface case. We need that statement for the proof of \autoref{prop:Extiso} below.

While some things become easier in the curve case due to the isomorphism $X^{[n]}\cong X^{(n)}$, we will see that the description of the extension groups of tautological bundles becomes actually more complicated in the curve case compared to the surface case. In the surface case, we have the simple formula
\[
 \Ext^*_{S^{[n]}}(E^{[n]},F^{[n]})\cong \Ext^*(E,F)\otimes S^{n-1}\Ho^*(\reg_X)\oplus \Ho^*(E^\vee)\otimes \Ho^*(E)\otimes S^{n-2}\Ho^*(\reg_X)\,\,;
\]
see \cite{KruExt, Krug--remarksMcKay}. In the curve case, the extension groups $\Ext^*(E^{[n]},F^{[n]})$ depend on more input data than just $\Ext^*(E,F)$ like certain Koszul cohomology groups. The reason for the complications in the curve case is that, while the functors $\Phi$ and $\Psi$ in the surface case are equivalences (though not mutually inverse), the functors $\pi^*$ and $\pi_*^{\sym_n}$ in the curve case are not equivalences (at least $\pi^*\colon \D(X^{(n)})\to \D_{\sym_n}(X^n)$ is fully faithful, but $\pi_*^{\sym_n}$ is not).   
\end{remark}

\begin{remark}
 Another result in \cite{Krug--remarksMcKay} is the formula $\Phi(\WW^k(L))\cong \wedge ^kL^{[n]}$, for every $L\in \Pic S$, where $\WW^k(L)=\Ind_{\sym_k\times \sym_{n-k}}^{\sym_n}\bigl((L^{\boxtimes k}\otimes \alt_k)\boxtimes \reg_S^{\boxtimes n-k}\bigr)$. Also this has an analogue in the curve case:
\begin{equation}\label{eq:WL}\pi_*^{\sym_n}\WW^k(L)\cong \wedge ^kL^{[n]}\quad\text{for $L\in \Pic X$.}\end{equation}
Again, one can either deduce this from the surface case by a base change argument, or imitate the proof from the surface case. The $k=n$ case of \eqref{eq:WL} can also be found in \cite[Prop.\ 3.2]{Sheridan--symmprod}. By \eqref{eq:WL} and \eqref{eq:IndDanila}, we get the formula
\[
 \Ho^*(\wedge^kL^{[n]})\cong \Ho^*_{\sym_n}(\WW^kL)\cong\Ho^*_{\sym_k\times\sym_{n-k}}\bigl((L^{\boxtimes k}\otimes \alt_k)\boxtimes \reg_S^{\boxtimes n-k}\bigr)\cong \wedge^k\Ho^*(L)\otimes S^{n-k}\Ho^*(\reg_X)\,.
\]
For the details of the computation behind the last two isomorphisms; see the proof of \cite[Prop.\ 4.1]{Krug--remarksMcKay}, or see the proof of \autoref{lem:dual} below for a similar computation. 
\end{remark}

\section{Extension groups of tautological bundles}\label{sect:technicalmain}

In this section, let $E,F\in \VB(X)$ be two vector bundles on a smooth projective curve $X$. 

\begin{prop}\label{prop:Extiso}
There is a natural isomorphism 
\[
\alpha\colon \Ext^*_{\sym_n}(C^\bullet_E, C^0_F)\xrightarrow{\cong} \Ext^*(E^{[n]},F^{[n]})  
 \]
such that the following diagram commutes:
\begin{align}\label{eq:cdExt}
\begin{CD}
\Ext_{\sym_n}^*(\CC^0_E,\CC^0_F)
@>{\pi_*^{\sym_n}}>>
\Ext^*(\pi_*^{\sym_n}\CC^0_E,\pi_*^{\sym_n}\CC^0_F)
\\
@V{\tau^*}VV
@VV{\widehat\beta}V
\\
\Ext_{\sym_n}^*(\CC^\bullet_E,\CC^0_F)
@>{\alpha}>>
\Ext_*(E^{[n]},F^{[n]})\,.
\end{CD} 
\end{align}
Here, $\tau \colon \CC^\bullet_E\to \CC^0_E$ is the truncation map, and $\hat\beta=\beta_F^{-1}\circ \_\circ \beta_E$ where $\beta\colon (\_)^{[n]}\xrightarrow\cong \pi_*^{\sym_n}\circ \CC$ is the isomorphism of functors of \autoref{prop:push}.
\end{prop}

\begin{proof}
By \autoref{prop:epsres}, the counit of adjunction $\eps_{\CC^0_E}\colon \pi^*\pi_*^{\sym_n}\CC^0_E\to \CC^0_E$ induces a quasi-isomorphism $\tilde \eps\colon \pi^*\pi_*^{\sym_n}\CC^0_E\to \CC^\bullet_E$. We define the isomorphism $\alpha$ as the composition
\begin{align*}
\Ext^*_{\sym_n}(\CC^\bullet_E, \CC^0_F)\xrightarrow{\tilde\eps^*} \Ext^*_{\sym_n}(\pi^*\pi_*^{\sym_n}\CC^0_E, C^0_F) \xrightarrow{\pi_*^{\sym_n}} \Ext^*(\pi_*^{\sym_n}\pi^*\pi_*^{\sym_n}\CC^0_E, \pi_*^{\sym_n}C^0_F)\\ \xrightarrow{(\eta_{\pi_*^{\sym_n} \CC_E^0})^*} \Ext^*(\pi_*^{\sym_n}\CC^0_E, \pi_*^{\sym_n}C^0_F)\xrightarrow{\widehat\beta} \Ext^*(E^{[n]}, E^{[n]})
\end{align*}
where $\eta\colon \id_{\D(X^{(n)})}\to \pi_*^{\sym_n}\pi^*$ is the unit of adjunction.
The first map of this composition is an isomorphism since $\tilde\eps$ is an quasi-isomorphism. The last map is an isomorphism since $\beta$ is an isomorphism. The middle part $\eta^*\circ \pi_*^{\sym_n}$ is the standard adjunction isomorphism. Hence, $\alpha$ is indeed an isomorphism.

The commutativity of the diagram \eqref{eq:cdExt} now follows from the identity $\tau\circ \tilde \eps=\eps_{\CC^0_E}$ and the general unit-counit identity $(\eta \pi_*^{\sym_n})\circ (\pi_*^{\sym_n} \eps)\cong \id_{\pi_*^{\sym_n}}$.
\end{proof}

\begin{prop}\label{prop:spectralseq}
 There is a spectral sequence 
\[
\IE^1_{p,q}=\Ext^q_{\sym_n}(C^{-p}_E,C^0_F) \quad\Longrightarrow \quad \IE^{p+q}= \Ext^{p+q}(E^{[n]},F^{[n]})\,.
\]
concentrated in the second quadrant. The differentials on the 1-level are given by 
\[
(d^{p*}_E)^{\sym_n} \colon \IE^1_{-p-1,q}=\Ext^q_{\sym_n}(\CC_E^{p+1},\CC^0_F)\to \IE^1_{-p,q}=\Ext^q_{\sym_n}(\CC_E^{p},\CC^0_F)  
\]
where $d^p_E\colon \CC^p_E\to \CC^{p+1}_E$ is the differential of the complex $\CC^\bullet_E$ as defined in \eqref{eq:ddef}.
The edge morphisms of the spectral sequence 
\[
\Ext^q(\CC^0_E,\CC^0_F)=\IE^1_{0,q} \twoheadrightarrow \IE^\infty_{0,q}\hookrightarrow \IE^q\cong \Ext^{q}(E^{[n]},F^{[n]}) 
\]
are given by $\widehat \beta\circ \pi_*^{\sym_n}$ 
\end{prop}

\begin{proof}
 There is the hyperext spectral sequence
\[
\IE^1_{p,q}=\Ext^q_{\sym_n}(C^{-p}_E,C^0_F) \quad\Longrightarrow \quad \IE^{p+q}= \Ext^{p+q}(C^\bullet_E,C^0_F)\,,
\]
whose edge morphisms are given by $\tau^*\colon \Ext^q_{\sym_n}(C^{0}_E,C^0_F) \to \Ext^{q}(C^\bullet_E,C^0_F)$. It is constructed as the spectral sequent associated to the double complex $\Hom^\bullet(C^\bullet_E,I^\bullet)$ by taking some injective $\sym_n$-equivariant resolution $I^\bullet$ of $C^0_F$.

Now, we replace $\IE^{p+q}= \Ext^{p+q}(C^\bullet_E,C^0_F)$ by $\IE^{p+q}=\Ext^{p+q}(E^{[n]},F^{[n]})$ using the isomorphism $\alpha$. Then, by \autoref{prop:Extiso}, the edge morphisms are of the desired form. 
\end{proof}

\subsection{A closer look at the terms of the spectral sequence}

We now express the terms  of the spectral $\IE^1_{p,q}=\Ext^q_{\sym_n}(C^{-p}_E,C^0_F)$ as products and sums of extension spaces of bundles on the curve $X$. We will first compute $(\IE^1_{p,q})^\vee$ and then apply Serre duality to get a formula for $\IE^1_{p,q}$.

The $\sym_n$-action on $X^n$ induces, by pull-back of $n$-forms, a $\sym_n$-linearisation on the canonical bundle $\omega_{X^n}$. We denote the canonical bundle equipped with this linearisation by $\omega_{[X^n/\sym_n]}$. This notation comes from the fact that it is the canonical line bundle of the associated quotient stack. Under the isomorphism $\omega_{X^n}\cong \omega_X^{\boxtimes n}$, the above linearisation differs from the one which permutes the box-factors by the sign-representation. In other words, $\omega_{[X^n/\sym_n]}\cong \omega_X^{\boxtimes n}\otimes \alt_n$; see \cite[Lem.\ 5.10]{KSequi}.

\begin{lemma}\label{lem:dual}
For every $p=0,\dots, n-1$, there are natural isomorphisms
\begin{align*}
&\Ext_{\sym_n}^*(\CC^p_E,\CC_F^0)^\vee[-n]\\\cong &\Ext^*_{\sym_n}(\CC^0_F,\CC^p_E\otimes \omega_X^{\boxtimes n}\otimes \alt_n)\\\cong& \left(\Ext^*(F, E\otimes \omega_X^{p+1})\otimes \wedge^{n-p-1}\Ho^*(\omega_X)\right)\oplus \left(\Ext^*(F,\omega_X)\otimes \Ho^*(E\otimes \omega_X^{p+1})\otimes \wedge^{n-p-2}\Ho^*(\omega_X)\right)\,.
\end{align*}
In the case $p=n-1$, the second summand vanishes.
\end{lemma}

\begin{proof}
 The first isomorphism is equivariant Serre duality; see e.g.\ \cite[Sect.\ 4.3]{BKR}. 
Looking at the definition of the terms of the complex $\CC^\bullet_E$ in \autoref{sect:push-pull}, we have 
 \begin{align}\label{eq:Extsum}
  \Ext^*(\CC^0_F,\CC^p_E\otimes \omega_X^{\boxtimes n}\otimes \alt_n)\cong \bigoplus_{\substack{i=1,\dots, n\\I\subset[n]\,,\, |I|=p+1}}\Ext^*(\pr_i^* F,E_I\otimes \alt_I\otimes \omega_X^{\boxtimes n}\otimes \alt_n)\,.
 \end{align}
The $\sym_n$-action on the Ext-space induced by the linearisations of $\CC^0_F$ and $\CC^p_E$ has the property that $g\in \sym_n$ maps the direct summand indexed by $(i,I)$ to the one indexed by $(g(i),g(I))$. Hence, the action on the index set of the direct sum \eqref{eq:Extsum} has two orbits, represented by $(1,[p+1])$ and $(1,[2,p+2])$ (except in the case $p=n-1$ where the second orbit does not exist). The stabilisers of these two representatives are 
\[G_1:=\sym_{[2, p+1]}\times \sym_{[p+2,n]}\quad,\quad  G_2:=\sym_{[2,p+2]}\times \sym_{[p+3,n]}\,.\]
Hence, by \eqref{eq:Danila}, the $\sym_n$-invariants of \eqref{eq:Extsum} are computed as 
\begin{align}
\Ext^*_{\sym_n}(\CC^0_F,\CC^p_E\otimes \omega_X^{\boxtimes n}\otimes \alt_n)\cong \begin{matrix}\Ext^*(\pr_1^* F,E_{[p+1]}\otimes \alt_{[p+1]}\otimes \omega_X^{\boxtimes n}\otimes \alt_n)^{G_1}\\\oplus\\ \Ext^*(\pr_1^* F,E_{[2,p+2]}\otimes \alt_{[2,p+2]}\otimes \omega_X^{\boxtimes n}\otimes \alt_n)^{G_2}\end{matrix}\label{eq:twosummands}
\end{align}
To compute the first direct summand, we note that
\[
 \sHom(\pr_1^* F,E_{[p+1]}\otimes \alt_{[p+1]}\otimes \omega_X^{\boxtimes n}\otimes \alt_n)\cong \sHom(F,E\otimes \omega_X^{p+1})_{[p+1]}\boxtimes \omega_{X}^{\boxtimes n-p-1}\otimes_{\alt_{[p+1,n]}}\,;
\]
compare \eqref{eq:EI} for the notation used.
As the $\sym_{[p+1]}$-action on $\sHom(F,E\otimes \omega_X^{p+1})_{[p+1]}$ is trivial, 
\begin{align*}
 \Ext^*(\pr_1^* F,E_{[p+1]}\otimes \alt_{[p+1]}\otimes \omega_X^{\boxtimes n}\otimes \alt_n)^{G_1}\cong& \Ho^*\bigl(\sHom(F,E\otimes \omega_X^{p+1})_{[p+1]}\boxtimes \omega_{X}^{\boxtimes n-p-1}\otimes_{\alt_{[p+1,n]}}  \bigr)^{\sym_{[p+2,n]}}\\
 \cong &\Ext^*(F,E\otimes \omega_X^{p+1})\otimes \wedge^{n-p-1}\Ho^*(\omega_{X})
\end{align*}
where the last isomorphism is due to the K\"unneth formula.
Similarly, the second direct summand of \eqref{eq:twosummands} is computed as 
\begin{align*}
&\Ext^*(\pr_1^* F,E_{[2,p+2]}\otimes \alt_{[2,p+2]}\otimes \omega_X^{\boxtimes n}\otimes \alt_n)^{G_2}
\\\cong&
\Ho^*\bigl(\sHom(F,\omega_X)
 \boxtimes (E\otimes \omega_X^{p+1})_{[2,p+2]}\boxtimes \omega_X^{\boxtimes n-p-2}\otimes \alt_{[p+3,n]}  \bigr)^{\sym_{[p+3,n]}}
\\\cong& \Ext^*(F,\omega_X)\otimes \Ho^*(E\otimes \omega_X^{p+1})\otimes \wedge^{n-p-2}\Ho^*(\omega_X)\,.\qedhere
 \end{align*}
\end{proof}

\begin{lemma}\label{lem:extterms} For $p=0,\dots, n-1$, we have
\begin{align*}
 &\Ext_{\sym_n}^*(C^p_E,C_F^0)[p]\\\cong& \left(\Ext^*(E,F\otimes\omega_X^{-p})\otimes S^{n-p-1}\Ho^*(\reg_X)\right)\oplus \left(\Ext^*(E,\omega_X^{-p})\otimes \Ho^*(F)\otimes S^{n-p-2}\Ho^*(\reg_X)\right)\,.   
\end{align*}
In the case $p=n-1$, the second summand vanishes.
\end{lemma}

\begin{proof}
We apply $(\_)^\vee[-(n-p)]$ to both sides of \autoref{lem:dual} and use \eqref{eq:Swedge} to get
\begin{align*}
\Ext_{\sym_n}^*(C^p_E,C_F^0)[p]
\cong  \begin{matrix}\bigl(\Ext^*(F, E\otimes \omega_X^{p+1})^\vee[-1]\bigr)\otimes \bigl(S^{n-p-1}(\Ho^*(\omega_X)^\vee[-1])\bigr)\\\oplus \bigl(\Ext^*(F,\omega_X)^\vee[-1]\bigr)\otimes \bigl(\Ho^*(E\otimes \omega_X^{p+1})^\vee[-1]\bigr)\otimes \bigl(S^{n-p-2}(\Ho^*(\omega_X)^\vee[-1])\bigr)\end{matrix}\,.
\end{align*}
Now, the formula follows by applying Serre duality on the curve $X$ to every factor of the right-hand side.
\end{proof}

\subsection{A closer look at the differentials of the spectral sequence}

Now, we describe the duals of the differentials on the 1-level of the spectral sequence of \autoref{prop:spectralseq}.

\begin{lemma}\label{lem:ddescription}
 Under the isomorphisms of \autoref{lem:dual}, the map 
 \[
  (d^p_{E*})^{\sym_n}\colon \Ext^*_{\sym_n}(\CC^0_F,\CC^p_E\otimes \omega_X^{\boxtimes n}\otimes \alt_n)\to \Ext^*_{\sym_n}(\CC^0_F,\CC^{p+1}_E\otimes \omega_X^{\boxtimes n}\otimes \alt_n) 
 \]
is given by 
$\begin{pmatrix}
A & B \\
0& D
\end{pmatrix}$ where 
\begin{align*}
A\colon \Ext^*(F, E\otimes \omega_X^{p+1})\otimes \wedge^{n-p-1}\Ho^*(\omega_X) \to \Ext^*(F, E\otimes \omega_X^{p+2})\otimes \wedge^{n-p-2}\Ho^*(\omega_X)\\
 \phi\otimes (t_{p+2}\wedge \dots\wedge t_n)  \mapsto \frac{-(p+1)}{n-p-1}\sum_{i=p+2}^n(-1)^{\alpha_i}(\phi\cup t_i)\otimes (t_{p+2}\wedge \dots\wedge \widehat t_i\wedge \dots\wedge t_n)  
\end{align*}
with $\alpha_i:=i+\sum_{j=p+2}^{i-1}\deg(t_j)\deg(t_i)$,
\begin{align*}
B\colon \Ext^*(F,\omega_X)\otimes \Ho^*(E\otimes \omega_X^{p+1})\otimes \wedge^{n-p-2}\Ho^*(\omega_X) &\to \Ext^*(F, E\otimes \omega_X^{p+2})\otimes \wedge^{n-p-2}\Ho^*(\omega_X)\\
 \theta\otimes s\otimes (t_{p+3}\wedge \dots\wedge t_n)  &\mapsto (\theta\cup s)\otimes (t_{p+3}\wedge \dots\wedge t_n),  
\end{align*}
\begin{align*}
&D\colon \Ext^*(F,\omega_X)\otimes \Ho^*(E\otimes \omega_X^{p+1})\otimes \wedge^{n-p-2}\Ho^*(\omega_X)\to \Ext^*(F,\omega_X)\otimes \Ho^*(E\otimes \omega_X^{p+2})\otimes \wedge^{n-p-3}\Ho^*(\omega_X)\\
 &\theta\otimes s\otimes (t_{p+3}\wedge \dots\wedge t_n)  \mapsto \frac{-(p+2)}{n-p-2}\sum_{i=p+3}^n(-1)^{\beta_i}\theta\otimes(s\cup t_i)\otimes (t_{p+3}\wedge \dots\wedge \widehat t_i\wedge \dots\wedge t_n)
\end{align*}
with $\beta_i:=i+\sum_{j=p+3}^{i-1}\deg(t_j)\deg(t_i)$. The second summand of 
$\Ext^*_{\sym_n}(\CC^0_F,\CC^{p+1}_E\otimes \omega_X^{\boxtimes n}\otimes \alt_n)$ vanishes for $p=n-2$, so does the map $D$.  
\begin{proof}
For $i\in [n]$ and $I\subset [n]$, we set 
\begin{align*}
 M_{(i,I)}:=\Ext^*(\pr_i^*F,E_I\otimes \omega_X^{\boxtimes n}\otimes \alt_{[n]\setminus I})\,\,,\quad 
M:=\Ext^*_{\sym_n}(\CC^0_F,\CC^p_E\otimes \omega_X^{\boxtimes n}\otimes \alt_n)&=\bigoplus_{|I|=p+1} M_{(i,I)}\\ 
N:=\Ext^*_{\sym_n}(\CC^0_F,\CC^{p+1}_E\otimes \omega_X^{\boxtimes n}\otimes \alt_n)&=\bigoplus_{|J|=p+2}M_{(j,J)}\,.
 \end{align*}
We write $f:=d^p_{E*}\colon M\to N$ and denote its components by $f_{(i,I),(j,J)}\colon M_{(i,I)}\to M_{(j,J)}$. Note that, by \eqref{eq:ddef}, 
\begin{equation}\label{eq:fvanish}
 f_{(i,I),(j,J)}=0\quad\text{unless $i=j$ and $I\subset J$.}
\end{equation}
As we already observed in the proof of \autoref{lem:extterms}, the group $G:=\sym_n$ acts on $M$, by conjugation by the $G$-linearisations of the $\CC^k_E$, and this action satisfies $g\cdot M_{(i,I)}=M_{g(i), g(I)}$ for every $g\in G$. 
Hence, \eqref{eq:Danila} gives isomorphisms
\[M^{\sym_n}\cong M_{(1,[p+1])}^{G_{(1,[p+1])}}\oplus M_{(1,[2,p+2])}^{G_{(1,[2,p+2])}}\quad,\quad
N^{\sym_n}\cong M_{(1,[p+2])}^{G_{(1,[p+2])}}\oplus M_{(1,[2,p+3])}^{G_{(1,[2,p+3])}}\,.\]
Under these isomorphisms, the map $f^{\sym_n}$ is given by 
\[
\begin{pmatrix}
A & B \\
C& D
\end{pmatrix} \colon M_{(1,[p+1])}^{G_{(1,[p+1])}}\oplus M_{(1,[2,p+2])}^{G_{(1,[2,p+2])}}\to M_{(1,[p+2])}^{G_{(1,[p+2])}}\oplus M_{(1,[2,p+3])}^{G_{(1,[2,p+3])}}
\]
where (compare \eqref{eq:Daninv} or \cite[App.\ A]{Sca1}) 
\begin{align*}
 &A(x)=\sum_{[g]\in \sym_n/G_{(1,[p+1])}}f_{g(1,[p+1]),(1,[p+2])}(gx) \,,\, B(y)=\sum_{[g]\in \sym_n/G_{(1,[2,p+2])}}f_{g(1,[2,p+2]),(1,[p+2])}(gy) \\
&C(x)=\sum_{[g]\in \sym_n/G_{(1,[p+1])}}f_{g(1,[p+1]),(1,[2,p+3])}(gx) \,,\, D(y)=\sum_{[g]\in \sym_n/G_{(1,[2,p+2])}}f_{g(1,[2,p+2]),(1,[2,p+3])}(gy). \\
 \end{align*}
It follows from \eqref{eq:fvanish} that all the summands of $C$ vanish. We now compute $A$ in detail, and leave the analogous computation of $B$ and $D$ to the reader. By \eqref{eq:fvanish}, the only non-vanishing summands of $A$ are the maps $f_{g(1,[p+1]),(1,[p+2])}$ with $g(1,[p+1])=(1,[p+2]\setminus\{k\})$ for $k=2,\dots, p+2$. For $[g]\in \sym_n/G_{(1,[p+1])}$ with $g(1,[p+1])=(1,[p+2]\setminus\{k\})$, a representative of $[g]$ with $g_{\mid [p+3,n]}=\id_{[p+3,n]}$, hence $g\in \sym_{[2,p+2]}$, can be chosen. Since the $\sym_{[2,p+2]}$-action on $M_{(1,[p+2])}$ is trivial, the $\sym_n$-equivariance of $f$ gives  
\begin{align}
A(x)=\sum_{[g]\in \sym_n/G_{(1,[p+1])}}f_{g(1,[p+1]),(1,[p+2])}(gy)&=\sum_{[g]\in \sym_n/G_{(1,[p+1])}}gf_{(1,[p+1]),(1,[p+2])}(y)\notag\\&=(p+1)f_{(1,[p+1]),(1,[p+2])}(y)\label{eq:multiplef}\,.
\end{align}
Now, we recall from the proof of \autoref{lem:dual} that  
\begin{align}
 M_{(1,[p+1])}^{G_{(1,[p+1])}}
 =&\Ext^*(\pr_1^*F,E_{[p+1]}\otimes \omega^{\boxtimes n}\otimes \alt_{[p+2,n]})^{\sym_{[2,p+1]}\times\sym_{[p+2,n]}}\notag\\
\cong&\Ext^*(F,E\otimes \omega^{p+1})\otimes (\Ho^*(\omega)^{\otimes n-p-1}\otimes \alt_{n-p-1})^{\sym_{n-p-1}}
\notag \\\cong& \Ext^*(F,E\otimes \omega^{p+1})\otimes \wedge^{n-p-1}\Ho^*(\omega)\label{eq:invaiso} \,.  
\end{align}
Under the last isomorphism, $\phi\otimes (t_{p+2}\wedge \dots\wedge t_n)$ corresponds to 
\begin{equation}\label{eq:tensorsum}\frac1{(n-p-1)!}\sum_{\sigma\in \sym_{[p+2,n]}}\sgn(\sigma) \phi\otimes \sigma(t_{p+2}\otimes \dots\otimes t_n)\end{equation}  
where $\sigma(t_{p+2}\otimes \dots\otimes t_n)$ is given by permuting the factors with an additional sign whenever to odd degree elements switch places, i.e.\ it is the action on the tensor power of the graded vector space $\Ho^*(\omega)$ as described in \autoref{subsect:graded}. 

For a fixed $i\in [p+2,n]$, every $\sigma\in \sym_{[p+2,n]}$ with $\sigma(i)=p+2$ can be written in the form $\sigma=\tau\circ (p+2,p+3,\dots, i)$ for a unique $\tau\in \sym_{[p+3,n]}$. Then,
\[
\sigma(t_{p+2}\otimes \dots\otimes t_n)= (-1)^{\sum_{j=p+2}^{i-1}\deg(t_j)\deg(t_i)}t_i\otimes\tau(t_{p+2}\otimes \dots\otimes \widehat t_i\otimes  \dots\otimes t_n) 
\]
and $\sgn(\sigma)=(-1)^{i-p-2}\sgn(\tau)$. Hence, $(n-p-1)!$ times \eqref{eq:tensorsum} can be rewritten as 
\begin{equation*}\sum_{i=p+2}^n\sum_{\tau\in \sym_{[p+3,n]}} (-1)^{i-p-2+\sum_{j=p+2}^{i-1}\deg(t_j)\deg(t_i)}\sgn(\tau) \phi\otimes t_i\otimes\tau(t_{p+2}\otimes \dots\otimes \widehat t_i\otimes  \dots\otimes t_n) \end{equation*}  
Since $f_{(1,[p+1]),(1,[p+2])}$ is $(-1)^{p+1}$ times the restriction map, see \eqref{eq:ddef}, we have
\[
f_{(1,[p+1]),(1,[p+2])} \bigl(\phi\otimes t_i\otimes\tau(t_{p+2}\otimes \dots\otimes \widehat t_i\otimes  \dots\otimes t_n)\bigr)=(-1)^{p+1}(\phi\cup t_i)\otimes \tau(t_{p+2}\otimes \dots\otimes \widehat t_i\otimes  \dots\otimes t_n)\,.
\]
Hence, applying $f_{(1,[p+1]),(1,[p+2])}$ to \eqref{eq:tensorsum} gives
\[
\frac{(-1)^{p+1}}{(n-p-1)!} \sum_{i=p+2}^n\sum_{\tau\in \sym_{[p+3,n]}} (-1)^{i-p-2+\sum_{j=p+2}^{i-1}\deg(t_j)\deg(t_i)}\sgn(\tau)(\phi\cup t_i)\otimes \tau(t_{p+2}\otimes \dots\otimes \widehat t_i\otimes  \dots\otimes t_n)\,.
\]
Noting that $\frac1{(n-p-1)!}=\frac1{n-p-1}\cdot \frac1{(n-p-2)!}$, we see that the above corresponds to 
\[
\frac{-1}{(n-p-1)} \sum_{i=p+2}^n (-1)^{i+\sum_{j=p+2}^{i-1}\deg(t_j)\deg(t_i)}(\phi\cup t_i)\otimes (t_{p+2}\wedge \dots\wedge \widehat t_i\wedge \dots\wedge t_n)\,.
\]
under the analogue of the isomorphism \eqref{eq:invaiso} for $M_{i,[p+2]}$. Combining this with \eqref{eq:multiplef} gives the assertion.
\end{proof}
\end{lemma}

\subsection{The functor $\CC$ on the level of $\Ext$}

\begin{lemma}
Under the isomorphism 
\[
 \Ext^*_{\sym_n}(\CC(E),\CC(F))= \Ext^*_{\sym_n}(\CC^0_E,\CC^0_F)\cong \begin{matrix}\Ext^*(E,F)\otimes S^{n-1}\Ho^*(\reg)\\\oplus \Ho^*(E^\vee)\otimes \Ho^*(F)\otimes S^{n-2}\Ho^*(\reg)\end{matrix}
\]
of \autoref{lem:extterms}, the map $\CC\colon \Ext^*(E,F)\to \Ext^*\bigl(\CC(E),\CC(F)\bigr)$ 
is given by
\[
 \CC(\phi)=(\phi\otimes \id^{n-1},0)\,.
\]
\end{lemma}

\begin{proof}
A class $\phi\in \Ext^i(E,F)$ corresponds to a morphism $E\to F[i]$ in $\D(X)$. We have $\CC(\phi)=\bigoplus_{i=1}^n \pr_i^*\phi\colon \CC(E)=\bigoplus_{i}\pr_i^*E\to \CC(F)[i]=\bigoplus_i\pr_i^*F[i]$. Now, we just need to follow the element $\CC(\phi)$ through the isomorphism \eqref{eq:Danila} and the K\"unneth isomorphism
 \begin{align*}
 \Ext^*_{\sym_n}(\CC(E),\CC(F))&\cong \Ext^*(\pr_1^*E,\pr_1^*F)^{\sym_{[2,n]}}\oplus \Ext^*(\pr_1^*E,\pr_2^*F)^{\sym_{[3,n]}}\\
 &\cong \bigl(\Ext^*(E,F)\otimes S^{n-1}\Ho^*(\reg)\bigr)\oplus \bigl(\Ho^*(E^\vee)\otimes \Ho^*(F)\otimes S^{n-2}\Ho^*(\reg)\bigr)
\end{align*}
to get the assertion.
\end{proof}

\begin{remark}\label{rem:remsuffice}
By \autoref{rem:mainsuffice}, in order to prove \autoref{thm:mainclosed} it suffices to prove the injectivity of $(\_)^{[n]}\colon \Ext^1(E,E)\to \Ext^1(E^{[n]},E^{[n]})$ for every stable bundle $E$ with $\mu(E)\notin[-1,n-1]$. For this, in turn, it suffices to prove that, for every $\phi\in \Ext^1(E,E)$, the element 
$\CC(\phi)=(\phi\otimes \id^{n-1},0)$ does not vanish under the edge morphism 
\[
\Ext^1\bigl(\CC(E),\CC(E)\bigr)\cong \IE^1_{0,1}\twoheadrightarrow \IE^\infty_{0,q}\hookrightarrow \IE^q\cong \Ext^{q}(E^{[n]},F^{[n]}) 
\]
In other words, we need that $\CC(\phi)=(\phi\otimes \id^{n-1},0)$ does not lie in the image of any of the differentials 
\[
\IE^1_{-1,1}\to \IE^1_{0,1}\,,\quad \IE^2_{-2,2}\to \IE^2_{0,1}\,,\quad \IE^3_{-3,3}\to \IE^3_{0,1}\,,\dots \,.
\]
The reason is that, up to the isomorphism $(\_)^{[n]}\cong \pi_*^{\sym_n}\circ \CC$, the edge morphism equals $\pi_*^{\sym_n}$; see \autoref{prop:spectralseq}. 
\end{remark}

\section{The case of genus $g\ge 2$}

In this section, let $X$ be a smooth projective curve of genus $g\ge 2$. We write $\IE$ for the $E=F$ case of the spectral sequence of \autoref{prop:spectralseq}.

\subsection{Results for simple bundles}

\begin{lemma}\label{lem:extvanish}
Let $E\in \VB(X)$ be a simple vector bundle. Then for all $p=1,\dots,n-1$, we have 
\[
\IE^1_{-p,p}\cong \Ext^p_{\sym_n}(C^p_E,C^0_E)=0 
\]
\end{lemma}

\begin{proof}
 By \autoref{lem:extterms}, we have
\begin{align}\label{eq:twosummands}
 \Ext_{\sym_n}^p(C^p_E,C_E^0)\cong \Hom(E,E\otimes\omega_X^{-p}) \oplus \left(\Hom(E,\omega_X^{-p})\otimes \Ho^0(E)\right)\,.   
\end{align}
Since $\omega_X^{-p}\subsetneq \reg_X$ and $E\otimes\omega_X^{-p}\subsetneq E$ are proper subbundles, non-vanishing of one of the two summands in \eqref{eq:twosummands} would yield an endomorphism of $E$ which is not just multiplication by a scalar. But we assumed $E$ to be simple.   
\end{proof}

\begin{theorem}\label{thm:simpleg2}
Let $E\in \VB(X)$ be a simple vector bundle. Then $E^{[n]}$ is again simple for all $n\in \IN$.
\end{theorem}

\begin{proof}
It follows from \autoref{lem:extvanish} that 
\[
\Hom(E^{[n]}, E^{[n]})\cong \IE^0\cong \IE^1_{0,0}\cong \Hom(E_E^0,C_E^0)\cong \Hom(E,E)\oplus \left(\Ho^0(E^\vee)\otimes \Ho^0(E)\right)\,.
\]
The second summand vanishes. Otherwise, we would get an endomorphism of $E$ which is not just multiplication by a scalar.
\end{proof}

\begin{prop}\label{prop:Extinjg2}
 Let $E\neq \reg_X$ be a simple sheaf. Then the map 
$(\_)^{[n]}\colon \Ext^1(E,E)\to \Ext^1(E^{[n]},E^{[n]})$ is injective for all $n\in \IN$.
\end{prop}

\begin{proof}
 This follows from \autoref{lem:extvanish} together with \autoref{rem:remsuffice}.
\end{proof}

As explained in \autoref{rem:mainsuffice}, this proves \autoref{thm:mainclosed} for $g\ge 2$.

\subsection{Stable bundles of positive degree}

In the following we will study the spectral sequence $\IE$ of \autoref{prop:spectralseq} in the special case that $g\ge 2$ and $F=E\neq \reg$ is a stable bundle. This will first be done for $\deg(E)\ge 0$ and later for $\deg(E)<0$. When we speak of two linear maps $f\colon A\to B$ and $f'\colon A'\to B'$ being \textit{isomorphic}, we mean that there are isomorphisms $A\cong A'$ and $B\cong B'$ which make the following diagram commutative:
\begin{equation}\label{eq:isodiag}
\begin{CD}
A
@>{f}>>
B
\\
@V{\cong}VV
@V{\cong}VV
\\
A'
@>{f'}>>
B'
\end{CD} 
\end{equation}
The following lemma describes the leftmost non-vanishing differential in every line of the 1-level of the spectral sequence $\IE$.

\begin{lemma}\label{lem:dd>0}
Let $\reg_X\neq E\in \VB(X)$ be a stable bundle with $\deg(E)\ge 0$.  The dual of the differential $\IE^1_{-1,2}\to \IE^1_{0,2}$ is isomorphic to 
\[
 (\cup\,,\,\cup)\colon \begin{matrix}\Hom(E, E\otimes \omega)\otimes\Ho^0(\omega)\\\oplus\\\Hom(E,\omega)\otimes\Ho^0(E\otimes \omega) \end{matrix} \to \Hom(E, E\otimes \omega^2)\qquad\text{for $n=2$,}
\]
and isomorphic to
 \[
\begin{pmatrix}
0&\cup  & \cup & 0 \\
0& 0& 0& \id_{\Ext^1(E,\omega)}\otimes \cup
\end{pmatrix}
  \colon  \begin{matrix}\wedge^2\Ho^0(\omega)\\\oplus\\ \Hom(E, E\otimes \omega)\otimes\Ho^0(\omega)\\\oplus\\\Hom(E,\omega)\otimes\Ho^0(E\otimes \omega)\\\oplus \\ \Ext^1(E,\omega)\otimes \Ho^0(E\otimes \omega)\otimes \Ho^0(\omega) \end{matrix}\longrightarrow \begin{matrix}\Hom(E, E\otimes \omega^2)\\\oplus\\ \Ext^1(E,\omega)\otimes \Ho^0(E\otimes \omega^2)\end{matrix}    
 \]
 for $n\ge 3$, where the various products of global sections are all denoted by the same symbol $\cup$. 
 For $p= 1, \dots, n-3$, the dual of the differential $\IE^1_{-p-1,2+p}\to \IE^1_{-p,2+p}$ is isomorphic to
 \[
\begin{pmatrix}
\cup &  0 & \cup \\
0&  \id_{\Ext^1(E,\omega)}\otimes \cup & 0
\end{pmatrix}
  \colon  \begin{matrix}\Hom(E, E\otimes \omega^{p+1})\otimes\Ho^0(\omega)\\\oplus \\ \Ext^1(E,\omega)\otimes \Ho^0(E\otimes \omega^{p+1})\otimes \Ho^0(\omega)\\\oplus\\\Hom(E,\omega)\otimes \Ho^0(E\otimes \omega^{p+1}) \end{matrix}\to \begin{matrix}\Hom(E, E\otimes \omega^{p+2})\\\oplus\\ \Ext^1(E,\omega)\otimes \Ho^0(E\otimes \omega^{p+2})\end{matrix} .   
 \]
The dual of the differential $\IE^1_{-n+1,n}\to \IE^1_{-n+2,n}$ is isomorphic to
 \[
\begin{pmatrix}
\cup  & \cup 
\end{pmatrix}
  \colon  \begin{matrix}\Hom(E, E\otimes \omega^{n-1})\otimes\Ho^0(\omega)\\\oplus\\\Hom(E,\omega)\otimes \Ho^0(E\otimes \omega^{n-1}) \end{matrix}\to \Hom(E, E\otimes \omega^{n})\,.   
 \]
\end{lemma}

\begin{proof}
 This is really just a special case of \autoref{lem:dual} and \autoref{lem:ddescription}. 
By the $p=0$ case of \autoref{lem:dual}, we have that $(\IE^{1}_{0,2})^\vee$ is the degree $n-2$ part of 
 \begin{equation}\label{eq:E0}
  \bigl(\Ext^*(E,E\otimes \omega)\otimes \wedge^{n-1}\Ho^*(\omega)\bigr) \oplus \bigl(\Ext^*(E,\omega)\otimes \Ho^*(E\otimes \omega)\otimes \wedge^{n-2}\Ho^*(\omega)\bigr)\,.
 \end{equation}
We note that $\Ho^1(E\otimes \omega)\cong \Ho^0(E^\vee)=0$, since $\reg\neq E$ is stable of non-negative degree. Furthermore, $\Ext^1(E,E\otimes \omega)\cong \Hom(E,E)$ and $\Ho^1(\omega)$ are both one-dimensional, which means that they can be omitted as tensor factors. Hence, the degree 2 part of \eqref{eq:E0} is, for $n\ge 3$, isomorphic to 
\[\wedge^2\Ho^0(\omega)\\\oplus\\ \Hom(E, E\otimes \omega)\otimes\Ho^0(\omega)\\\oplus\\\Hom(E,\omega)\otimes\Ho^0(E\otimes \omega)\\\oplus \\ \Ext^1(E,\omega)\otimes \Ho^0(E\otimes \omega)\otimes \Ho^0(\omega) \]
which is the source of the dual of the differential $\IE_{-1,2}^1\to \IE_{0,2}^1$ as asserted in the lemma. The other sources and targets asserted follow analogously from \autoref{lem:dual}.

Now, the description of the maps comes from \autoref{lem:ddescription}. To illustrate this, let us consider in detail the last component $\Ext^1(E,\omega)\otimes \Ho^0(E\otimes \omega)\otimes \Ho^0(\omega)\to \Ext^1(E,\omega)\otimes\Ho^0(E\otimes \omega^2)$ of the dual of $\IE_{-1,2}^1\to \IE_{0,2}^1$ which is asserted to be $\id_{\Ext^1(E,\omega)}\otimes \cup$. This component is given by the $p=0$ case of the map $D$ of \autoref{lem:ddescription}. Note that we omitted the one dimensional space $\Ho^1(\omega)=\langle t\rangle$ in the notation, which means that $\theta\otimes s\otimes u\in \Ext^1(E,\omega)\otimes \Ho^0(E\otimes \omega)\otimes \Ho^0(\omega)$ corresponds to the element 
\begin{equation}\label{eq:nonzero}
 \theta\otimes s\otimes u\wedge t\wedge \dots\wedge t
\end{equation}
in the description of $D$. Even though $t$ occurs multiple times in the wedge product, the expression \eqref{eq:nonzero} is usually non-vanishing. The reason is that $t$ is of cohomological degree $1$, so the graded wedge product $t\wedge \dots\wedge t$ is really the symmetric power $t^{n-3}$; see \eqref{eq:Swedge}. Now, by \autoref{lem:ddescription} and the fact that $s\cup t=0$ for degree reasons, the element \eqref{eq:nonzero} is sent to 
\[
 \frac{2}{n-2}\theta\otimes(s\cup u)\otimes (t\wedge \dots\wedge t)\,,
\]
which, after hiding the factor $\Ho^1(\omega)$ again, is $\frac{2}{n-2}\theta\otimes(s\cup u)\in \Ext^1(E,\omega)\otimes\Ho^0(E\otimes \omega^2)$. We see that our component is given by $\frac{2}{n-2} \id_{\Ext^1(E,\omega)}\otimes \cup$, and the factor $\frac{2}{n-2}$ can be absorbed by our choice of the component of the vertical isomorphisms as in \eqref{eq:isodiag}. All the other components of the duals of the differentials can be read off from \autoref{lem:ddescription} analogously. 
\end{proof}

For $F\in \VB(X)$, we consider the cokernel of the multiplication map
\[
 \K_{0,2}(F,\omega):=\coker\left(\Ho^0(F\otimes \omega)\otimes \Ho^0(\omega) \xrightarrow\cup \Ho^0(F\otimes \omega^2)\right)\,. 
\]
If $F=\reg$, we simply write $\K_{0,2}(\omega):=\K_{0,2}(\reg,\omega)$. This is a special case of \emph{Koszul cohomology}; see e.g.\ \cite{Green--Koszulcoh}, \cite{Aprodu-Nagel--Koszbook} for general information on this topic.
Famously, there is a close relationship between Koszul cohomology of line bundles on $X$ and global sections of tautological bundles, which was used to prove the Gonality Conjecture; see \cite{Voisin--Green1, Voisin--Green2, Ein-Lazarsfeld--gonalityconj}. Hence, it should not come as a surprise that Koszul cohomology also shows up in our computations of extension groups of tautological bundles.

Let $\overline \cup \colon \Hom(E,\omega)\otimes \Ho^0(E\otimes \omega)\to \K_{0,2}(\sEnd(E),\omega)$ be the composition of the cup product followed by the quotient map $\Hom(E,E\otimes \omega^2)\to \K_{0,2}(\sEnd(E),\omega)$. We set 
\begin{equation}\label{eq:Wdef}
 W_E:=\coker\Bigl(\overline \cup \colon \Hom(E,\omega)\otimes \Ho^0(E\otimes \omega)\to \K_{0,2}(\sEnd(E),\omega)  \Bigr)\,.
\end{equation}

\begin{lemma}\label{lem:Etermsd>0}
 Let $\reg_X\neq E\in \VB(X)$ be a stable bundle with $\deg(E)\ge 0$. Then
 \begin{enumerate}
  \item $\IE_{0,1}^\infty\cong \Ext^1(E,E)\oplus \Ho^1(\reg_X)\oplus\bigl(\Ho^0(E)\otimes \Ho^1(E^\vee)  \bigr)$,
  \item 
    $\IE_{-1,2}^\infty\cong\begin{cases}W_E^\vee \quad&\text{for $n=2$,}\\ W_E^\vee \oplus \bigl(\Ho^0(E)\otimes \K_{0,2}(E,\omega)^\vee\bigr)\quad&\text{for $n\ge 3$.}\end{cases}$
  \item If $g\ge 3$, then $\IE^\infty_{-p-1, 2+p}\cong 0$ for all $p\ge 1$. If $g=2$, then $\IE^\infty_{-p-1, 2+p}\cong 0$ for all $p\ge 2$.  
\end{enumerate}
\end{lemma}

\begin{proof}
By \autoref{lem:extvanish}, we have $\IE_{0,1}^\infty=\IE_{0,1}^1$. Thus, (i) follows by \autoref{lem:extterms}.

By \autoref{lem:extvanish}, we also have $\IE_{-1,2}^\infty=\IE_{-1,2}^2=\ker\bigl(d\colon\IE_{-1,2}^1\to  \IE_{0,2}^1 \bigr)$.
Thus (ii) follows from the first part of \autoref{lem:dd>0}. 

For (iii), it suffices to proof the injectivity of the differential $\IE^1_{-p-1, 2+p}\to \IE^1_{-p, 2+p}$ or, equivalently, the surjectivity of its dual. Both multiplication maps,  \[\Hom(E,E\otimes \omega^{p+1})\otimes \Ho^0(\omega)\to \Hom(E,E\otimes \omega^{p+2})\quad,\quad \Ho^0(E\otimes \omega^{p+1})\otimes \Ho^0(\omega)\to \Ho^0(E\otimes \omega^{p+2})\] are surjective for every $p\ge 1$ (every $p\ge 2$ if $g=2$) by \cite[Prop.\ 2.2]{Butler--normalgen}: For the first map, set $F=\sHom(E,E\otimes \omega^{p+1})$ and $E=\omega$ in \textit{loc.\ cit.}, and for the second $F=E\otimes \omega^{p+1}$ and $E=\omega$. The surjectivity of the dual of $\IE^1_{-p-1, 2+p}\to \IE^1_{-p, 2+p}$ follows now from the second part of \eqref{lem:dd>0}. 
\end{proof}

\begin{remark}
The condition $g\ge 3$ is really necessary for the vanishing of $\IE^2_{-2,3}$ if $n\ge 3$.
Let $E=L$ be a line bundle on a curve of genus $g=2$ with $\Hom(L,\omega)\cong \Ho^1(L)=0$. Then $\IE^2_{-2,3}\neq 0$ as the multiplication map $\Ho^0(\omega^2)\otimes \Ho(\omega)\to \Ho^0(\omega^3)$ is not surjective (its cokernel is one-dimensional; see the diagram \eqref{eq:cupdiag} below). 
\end{remark}

\begin{theorem}\label{Thm:Ext1}
Let $\reg_X\neq E\in \VB(X)$ be a stable bundle with $\deg(E)\ge 0$ on a curve of genus $g\ge 3$. Then
there is a short exact sequence of vector spaces
\[
 0\to \Ext^1(E,E)\oplus \Ho^1(\reg_X)\oplus\bigl(\Ho^0(E)\otimes \Ho^1(E^\vee)  \bigr)\to \Ext^1(E^{[n]},E^{[n]})\to K\to 0
\]
with 
\[K\cong \begin{cases} W_E^\vee \quad&\text{for $n=2$,}\\
  W_E^\vee \oplus \bigl(\Ho^0(E)\otimes \K_{0,2}(E,\omega_X)^\vee\bigr) \quad&\text{for $n\ge 3$.}\end{cases}\]
The $n=2$ case is also valid on a curve of genus $g=2$.
\end{theorem}

\begin{proof}
 By part (iii) of \autoref{lem:Etermsd>0}, we have a short exact sequence
 \[
  0\to \IE_{0,1}^\infty\to \Ext^1(E^{[n]},E^{[n]})\to \IE_{-1,2}^\infty \to 0
 \]
The outer terms of this sequence are described by part (i) and (ii) of \autoref{lem:Etermsd>0}.
\end{proof}

\begin{remark}\label{rem:Kvanish}
The cokernel $K$ in the above description of $\Ext^1(E^{[n]},E^{[n]})$ often vanishes. If $\deg E\ge 3$, then $\K_{0,2}(E,\omega_X)=0$. In other words, the product $\cup \colon \Ho^0(E\otimes \omega)\otimes \Ho^0(\omega)\to \Ho^0(E\otimes \omega^2)$ is surjective, as follows form \cite[Prop.\ 2.2]{Butler--normalgen} if we set $F=E\otimes \omega$ and $E=\omega$ in \textit{loc.\ cit.}. 

If $E=L$ is a line bundle, and $X$ is not hyperelliptic, then we have $\K_{0,2}(E,\omega_X)=0$ already for $\deg L\ge 2$; see \cite[Thm.\ 1]{Butler--line}.

The vector space $W_E$ is a quotient of $\K_{0,2}(\sEnd(E), \omega_X)$. If $E=L$ is a line bundle, we have $\sEnd(L)=\reg_X$, hence $\K_{0,2}(\sEnd(E), \omega_X)= \K_{0,2}(\omega_X)$. The latter vector space vanishes if $X$ is non-hyperelliptic by the Max Noether Theorem; see e.g.\ \cite[p.\ 117]{CurvesBookI}. We also have $\K_{0,2}(\omega_X)=0$ if $g=2$; see \autoref{rem:hyperell} below. 
\end{remark}

\begin{cor}
Let $X$ be non-hyperelliptic. Then, for every line bundle $L$ on $X$ of degree $\deg(L)\ge 2$, we have
\begin{equation}\label{eq:ExtL}
 \Ext^1(L^{[n]},L^{[n]})\cong \Ext^1(L,L)\oplus \Ho^1(\reg_X)\oplus\bigl(\Ho^0(L)\otimes \Ho^1(L^\vee)  \bigr)\,.
\end{equation}
In particular, for any two line bundles $L_1, L_2\in \Pic(X)$ of the same degree $\deg(L_1)=\deg(L_2)\ge 2$ and every $n\ge 2$, we have 
\begin{equation}\label{eq:BNineq}
 \ho^0(L_1)<\ho^0(L_2)\quad \iff \quad \ext^1(L_1^{[n]},L_1^{[n]})< \ext^1(L_2^{[n]},L_2^{[n]})\,.
\end{equation}
\end{cor}

As mentioned in the introduction, this can be interpreted as follows: Let $X$ be a non-hyperelliptic curve, and let $n\le d\le2(g-1)$. Then, under the closed embedding $\phi\colon \Pic_d(X)\hookrightarrow \cM_{X^{(n)}}$, $[E]\mapsto [E^{[n]}]$, the stratification on $\cM_{X^{(n)}}$ given by the dimension of the tangent spaces pulls back to the Brill--Noether stratification on $\Pic_d(X)$.

\begin{remark}
Denote the first Chern class of $L^{[n]}$ for any degree $d$ line bundle $L$ on $X$ by $\cc_1(d,n)$. Then, there is an isomorphism 
\[
\Pic_d(X)\xrightarrow \cong \Pic_{\cc_1(d,n)}(X^{(n)})\quad ,\quad [L]\mapsto [\det L^{[n]}]\,;
\]
see \cite[Sect.\ 3.2]{Sheridan--symmprod}.
Hence, the embedding $\phi\colon \Pic_d(X)\to \cM_{X^{(n)}}$, precomposed by the inverse of the above isomorphism, is a section of $\det\colon \cM'_{X^{(n)}}\to \Pic_{\cc_1(d,n)}(X)$ where $\cM'_{X^{(n)}}\subset \cM_{X^{(n)}}$ denotes the component containing the image of $\phi$. 
\end{remark}

\begin{remark}
The equivalence \eqref{eq:BNineq} still holds on a non-hyperelliptic curve $X$ if $\deg(L_1)=\deg(L_2)=1$. The only difference is that, in \eqref{eq:ExtL}, there is the additional summand $\Ho^0(L)\otimes \K_{0,2}(L,\omega)^\vee$, which is $1$-dimensional for $\deg(L)=1$ and $\ho^0(L)=1$; compare diagram \eqref{eq:diagcup2} below.
\end{remark}

\begin{remark}\label{rem:hyperell}
 For a hyperelliptic curve $X$ of genus $g\ge 3$ and $\deg(L_1)=\deg(L_2)>0$, we can still proof the implication
\[
 \ho^0(L_1)<\ho^0(L_2)\quad \implies \quad \ext^1(L_1^{[n]},L_1^{[n]})< \ext^1(L_2^{[n]},L_2^{[n]})\,.
\]
To see this, we first compute $\K_{0,2}(\omega)=\coker\bigl(\Ho^0(\omega)\otimes \Ho^0(\omega)\xrightarrow\cup \Ho^0(\omega^2)\bigr)$ for $X$ a hyperelliptic curve of genus $g$. Let $h\colon X\to \IP^1$ be a double cover and $L:=h^*\reg_{\IP^1}(1)$ (which is usually denoted by $L=\reg(g_2^1)$ in the literature). Let $V:=\Ho^0(\IP^1,\reg(1))\cong \IC^2$. Then, for every $a\in \IN$, the pull back 
\[
 h^*\colon S^aV\cong \Ho^0(\IP^1,\reg(a))\to \Ho^0(X,L^n)
\]
is injective, and a dimension count shows that it is an isomorphism if and only if $a\le g$. For all $a,b\in \IN$, we get a commutative diagram
\begin{equation}\label{eq:cupdiag}
\begin{CD}
\Ho^0(X,L^a)\otimes \Ho^0(X,L^b)
@>{\cup}>>
\Ho^0(X,L^{a+b})
\\
@A{h^*\otimes h^*}AA
@A{h^*}AA
\\
S^aV\otimes S^bV
@>{\cup}>>
S^{a+b}V
\end{CD} 
\end{equation}   
where the bottom arrow is surjective and the two upwards arrows are injective. We have $\omega_X\cong L^{g-1}$. Hence, considering diagram \eqref{eq:cupdiag} for $a=g-1=b$, in which case the left upward arrow is an isomorphism, gives
\[
 \knum_{0,2}(\omega)=\ho^0(\omega^2)-\dim(S^{2(g-1)}V)=3(g-1)- 2(g-1)-1=g-2\,.
\]
In particular, $\K_{0,2}(\omega)=0$ for a curve of genus $2$.

Now, let us proof that $\ext^1(L_1^{[n]},L_1^{[n]})< \ext^1(L_2^{[n]},L_2^{[n]})$ for $d=\deg(L_1)=\deg(L_2)>0$ and $\ho^0(L_2)=\ho^0(L_1)+1$. By \autoref{Thm:Ext1}, for every line bundle $L$ of positive degree we have a (non-canonical) direct sum decomposition of vector spaces
\begin{equation}\label{eq:Extsum}
 \Ext^1(L^{[n]},L^{[n]})\cong \Ho^1(\reg_X)\oplus \Ho^1(\reg_X)\oplus\bigl(\Ho^0(L)\otimes \Ho^1(L^\vee)  \bigr) \oplus W_L^\vee \oplus \bigl(\Ho^0(L)\otimes \K_{0,2}(L,\omega_X)^\vee\bigr)
\end{equation}
where the last summand vanishes in the $n=2$ case.

The dimension of the third summand $\Ho^0(L)\otimes \Ho^1(L^\vee)$ of \eqref{eq:Extsum} grows by $\ho^1(L^\vee)=d+g-1$ when we pass from $L_1$ to $L_2$. The dual $W_L$ of the fourth summand of \eqref{eq:Extsum} is a quotient of $\K_{0,2}(\omega)$. Hence, its dimension can shrink by at most $\knum_{0,2}(\omega)=g-2$ when we pass from $L_1$ to $L_2$. Hence, it suffices to proof that the fifth summand $\Ho^0(L)\otimes \K_{0,2}(L,\omega_X)^\vee$ cannot shrink when we pass from $L_1$ to $L_2$. This summand is zero if $\Ho^0(L)=0$ and also if $d\ge 3$; see \autoref{rem:Kvanish}. Hence, the only case to consider is $d=2$, $\ho^0(L_1)=1$, and $\ho^0(L_2)=2$. Let $0\neq s\in\Ho^0(L_1)$. Then, we have the commutative diagram
\begin{equation}\label{eq:diagcup2}
\begin{CD}
\Ho^0(L_1\otimes \omega)\otimes \Ho^0(\omega)
@>{\cup}>>
\Ho^0(L_1\otimes \omega^2)
\\
@A{s\cup }AA
@A{s\cup}AA
\\
\Ho^0(\omega)\otimes \Ho^0(\omega)
@>{\cup}>>
\Ho^0(\omega^2)
\end{CD} 
\end{equation}   
where the left vertical map is an isomorphism, and the right vertical map is injective with image of codimension 1. Hence $\knum_{0,2}(L_1,\omega)=\knum_{0,2}(\omega)+1=g-1$. It follows that the fifth summand in \eqref{eq:Extsum} for $L_1$ has dimension at most $g-1$. 
To compute $\knum_{0,2}(L_2,\omega)$, we note that $L_2=L=\reg(g_2^1)$. Hence, we can consider the $a=g$ and $b=g-1$ case of \eqref{eq:cupdiag}. In this case, the left vertical arrow is still an isomorphism. Hence
\[
 \knum_{0,2}(L_2,\omega)=\ho^0(L^{2g-1})-\dim S^{2g-1} V= 4g-2-g+1-2g= g-1\,,
\]
and the fifth summand in \eqref{eq:Extsum} is bigger for $L=L_2$ than for $L=L_1$.
\end{remark}

\subsection{Stable bundles of negative degree}

We now come to the case of stable bundles of negative degree.

\begin{lemma}\label{lem:ddescription<0}
Let $\reg_X\neq E\in \VB(X)$ be a simple bundle with $\deg(E)< 0$.  The dual of the differential $\IE^1_{-1,2}\to \IE^1_{0,2}$ is isomorphic to
\[
 (\cup\,,\, \cup)\colon \begin{matrix}\Hom(E, E\otimes \omega)\otimes\Ho^0(\omega)\\\oplus\\ \Hom(E,\omega)\otimes \Ho^0(E\otimes \omega)\end{matrix} \to \Hom(E, E\otimes \omega^2)\qquad\text{for $n=2$,}
\]
and isomorphic to 
 \[
\begin{pmatrix}
0 & \cup & 0 & \cup \\
0& 0& \id\otimes \cup & \id\otimes \delta
\end{pmatrix}
  \colon  \begin{matrix}\wedge^2\Ho^0(\omega)\\\oplus \\\Hom(E, E\otimes \omega)\otimes\Ho^0(\omega)\\\oplus\\    \Hom(E,\omega)\otimes \Ho^1(E\otimes \omega)\otimes \Ho^0(\omega)  \\\oplus\\ \Hom(E,\omega)\otimes \Ho^0(E\otimes \omega)\end{matrix}\longrightarrow \begin{matrix}\Hom(E, E\otimes \omega^2)\\\oplus\\ \Hom(E,\omega)\otimes \Ho^1(E\otimes \omega^2)\end{matrix}    
 \]
 for $n\ge 3$, where $\delta$ is the product with a generator of $\Ho^1(\omega)$ (note, however, that the only case when $\Ho^0(E\otimes \omega)$ and $\Ho^1(E\otimes \omega^2)$ are both non-vanishing is $E=\omega^{-1}$, in all other cases $\delta=0$). 
 For $p=1,\dots, n-3$, the dual of the differential $\IE^1_{-p-1,2+p}\to \IE^1_{-p,2+p}$ is isomorphic to
 \[
\begin{pmatrix}
\cup &  0 & \cup\\
0&  \id\otimes \cup & \id\otimes \delta 
\end{pmatrix}
  \colon  \begin{matrix}\Hom(E, E\otimes \omega^{p+1})\otimes\Ho^0(\omega)\\\oplus \\ \Hom(E,\omega)\otimes \Ho^1(E\otimes \omega^{p+1})\otimes \Ho^0(\omega)\\\oplus \\ \Hom(E,\omega)\otimes \Ho^0(E\otimes \omega^{p+1}) \end{matrix}\longrightarrow \begin{matrix}\Hom(E, E\otimes \omega^{p+2})\\\oplus\\ \Hom(E,\omega)\otimes \Ho^1(E\otimes \omega^{p+2})\end{matrix} \,.   
 \]
 Again, $\delta$ is multiplication by a generator of $\Ho^1(\omega)$, but either its domain or its codomain vanishes except for $E=\omega^{-(p+1)}$.\\ The dual of the differential $\IE^1_{-n+1,n}\to \IE^1_{-n+2,n}$ is isomorphic to
 \[
\begin{pmatrix}
\cup & \cup 
\end{pmatrix}
  \colon  \begin{matrix}\Hom(E, E\otimes \omega^{n-1})\otimes\Ho^0(\omega)\\\oplus \\ \Hom(E,\omega)\otimes \Ho^0(E\otimes \omega^{n-1}) \end{matrix}\longrightarrow \Hom(E, E\otimes \omega^{n}) \,.   
 \]
\end{lemma}

\begin{proof}
This is analogous to \autoref{lem:dd>0}. Note that now $\Ext^1(E,\omega)\cong \Ho^0(E)=0$ vanishes. 
\end{proof}

\begin{lemma}\label{lem:Einfty<0}
 Let $\reg_X\neq E\in \VB(X)$ be a simple bundle with $\deg(E)< 0$ and $n\ge 2$. Then
 \begin{enumerate}
  \item $\IE_{0,1}^\infty\cong \Ext^1(E,E)\oplus \Ho^1(\reg_X)\oplus\bigl(\Ho^1(E)\otimes \Ho^0(E^\vee)  \bigr)$.
  \item With $W_E$ as in \eqref{eq:Wdef}, we have $\IE^{\infty}_{-1,2}\cong W_E^\vee$ for every $n\ge 2$.
  \item If $g\ge 3$, then $\IE^\infty_{-p-1, 2+p}\cong 0$ for all $p\ge 1$. If $g=2$, then $\IE^\infty_{-p-1, 2+p}\cong 0$ for all $p\ge 2$.
 \end{enumerate}
\end{lemma}

\begin{proof}
The first part is analogous to \autoref{lem:Etermsd>0}. 
Let now $0\neq s\in \Ho^0(\omega)$, and consider the short exact sequence
\[
 0\to E\otimes \omega^{p+1}\xrightarrow{s} E\otimes \omega^{p+2}\to \reg_K^{\oplus r}\to 0\,.
\]
where $K=Z(s)\subset X$ is a canonical divisor. Since $\Ho^1(\reg_K)=0$, the map \[\_\cup s\colon \Ho^1(E\otimes \omega^{p+1})\to \Ho^1(E\otimes \omega^{p+2})\] is surjective. It follows that \[\cup \colon \Ho^1(E\otimes \omega^p)\otimes\Ho^0(\omega)\to \Ho^1(E\otimes \omega^{p+2})\] is surjective. 
This shows that the dual of the differential $\IE^1_{-p-1,2+p}\to \IE^1_{-p,2+p}$, as described in \autoref{lem:ddescription<0}, surjects to the second direct summand of $(\IE^1_{-p,2+p})^\vee$. The case $p=0$ proves the second part of our lemma.
For the third part, we recall from the proof of \autoref{lem:Etermsd>0} that, by \cite[Prop.\ 2.2]{Butler--normalgen}, the 
multiplication map \[\Hom(E, E\otimes \omega^{p+1})\otimes\Ho^0(\omega)\to \Hom(E,E\otimes \omega^{p+2})\] is  surjective for $p\ge 1$ ($p\ge 2$ if $g=2$). Hence, the dual of the differential $\IE^1_{-p-1,2+p}\to \IE^1_{-p,2+p}$ 
is surjective. 
\end{proof}

\begin{theorem}\label{Thm:Ext1d<0}
Let $\reg_X\neq E\in \VB(X)$ be a stable bundle with $\deg(E)< 0$. Then there is a short exact sequence
\begin{align*}
0\to  \Ext^1(E,E)\oplus \Ho^1(\reg_X)\oplus\bigl(\Ho^1(E)\otimes \Ho^0(E^\vee)  \bigr)  \to  \Ext^1(E^{[n]},E^{[n]})\to W^\vee_E\to 0\,.
\end{align*}
\end{theorem}

\begin{proof}
 This follows from \autoref{lem:Einfty<0} the same way as \autoref{Thm:Ext1} followed from \eqref{lem:Etermsd>0}. 
\end{proof}

\subsection{Singular points in the moduli space}

Let $E\neq \reg_X$ be a stable vector bundle. 
Our above computations show that, in almost all cases, $\IE^\infty_{0,2}\subset \Ext^2(E^{[n]},E^{[n]})$ is strictly bigger than $\Ho^2(\reg_{X^{(n)}})\cong \wedge^2\Ho^1(\reg_X)$, which means that the obstruction space 
\[
\Ext^2_0(E^{[n]},E^{[n]}):=\ker\Bigl( \Ext^2(E^{[n]},E^{[n]})\xrightarrow{\tr} \Ho^2(\reg_{X^{(n)}})  \Bigr) 
\]
does not vanish. However, for $n=2$, there are a few exceptions where $\Ext^2_0(E^{[2]},E^{[2]})=0$.

\begin{prop}\label{prop:smooth}
Let $X$ be either of genus $g=2$ or non-hyperelliptic of genus $g=3$ and let $L$ be a line bundle on $X$ with $\Hom(L,\omega)\otimes \Ho^0(L\otimes \omega)=0$. Then $\Ext^2_0(L^{[2]},L^{[2]})=0$. In particular $[L^{[2]}]$ is a smooth point of $\cM_{X^{(2)}}$.  
\end{prop}

\begin{proof}
 Since the trace map has the right-inverse $\alpha\mapsto \alpha\cup \id_L$, it suffices to show that $\ext^2(L^{[2]},L^{[2]})=\ho^2(\reg_{X^{(2)}})=\binom g2$. By the vanishing of $\Hom(L,\omega)\otimes \Ho^0(L\otimes \omega)$ and the description of the differential $\IE^1_{-1,2}\to \IE^1_{0,2}$ in \autoref{lem:dd>0} and \autoref{lem:ddescription<0}, we have
\[
 \ext^2(L^{[2]},L^{[2]})=\dim(\IE^2_{0,2})\ge\ho^0(\omega)^2-\ho^0(\omega^2)=g^2-3(g-1)\,.  
 \]
with equality in the middle if and only if $\cup\colon \Ho^0(\omega)\otimes \Ho^0(\omega)\to \Ho^0(\omega^2)$ is surjective, which is the case if and only if $X$ is non-hyperelliptic or $g=2$. Furthermore, we have $g^2-3(g-1)=\binom g2$ if and only if $g=2$ or $3$.
\end{proof}

\begin{remark}\label{rem:smooth}
 For $|d|=|\deg(L)|\gg 0$, the condition $\Hom(L,\omega)\otimes \Ho^0(L\otimes \omega)=0$ is always satisfied. Hence, for $X$ 
of genus $g=2$ or non-hyperelliptic of genus $g=3$, and $|d|\gg 0$, the image of the map $\phi\colon \Pic_d(X)\hookrightarrow \cM_{C^{(2)}}$ is contained in the smooth locus. Since in most other cases the image of $(\_)^{[n]}$ is contained in the singular locus of $\cM_{X^{(n)}}$ (see \autoref{thm:sing} below), one might conjecture that the whole component of $\cM_{X^{(2)}}$ containing $\phi(\Pic_d(X))$ for $|d|\gg 0$ and $X$ of genus $g=2$ or non-hyperelliptic of genus $g=3$ is smooth. 
 \end{remark}

\begin{theorem}\label{thm:sing}
Let $E\in \VB(X)$ be a stable vector bundle on a curve of genus $g\ge 3$. For $|\mu(E)|\gg 0$, the point $[E^{[n]}]$ is a singular point of $\cM_{X^{(n)}}$ for every $n\in \IN$, except if $n=2$, and $X$ is non-hyperelliptic of genus $g=3$.
\end{theorem}

\begin{proof}
For $n=2$, the component of $\cM_{X^{(2)}}$ containing $E^{[2]}$ is irreducible of the expected dimension if $|\mu(E)|\gg 0$; see \cite[Cor.\ 1.3]{Krugdisc}.
But, by a generalisation of the computation of the proof of \autoref{prop:smooth}, we have
\begin{align*}
 \ext^2(E^{[2]},E^{[2]})\ge \hom(E,E\otimes \omega)\ho^0(\omega)-\hom(E,E\otimes \omega^2)&=\bigl(r^2(g-1)+1\bigr)g-3r^2(g-1)\\&=(g-3)r^2(g-1)+g
\end{align*}
which for $g\ge 4$ is bigger than $\ho^2(\omega_{X^{(2)}})=\binom g2$. Hence, for $g\ge 4$, we have
$\ext^2_0(E^{[2]},E^{[2]})>0$. This means that the tangent space $\TT_{\cM_{X^{(2)}}}([E^{[2]}])$ is of higher dimension than the expected dimension, which is the dimension of $\cM_{X^{(2)}}$ in the point $[E^{[2]}]$.  

For $n\ge 3$, we note that it suffices to prove the non-vanishing of the Yoneda square map 
\[
 q\colon \Ext^1(E^{[n]},E^{[n]})\to \Ext^2(E^{[n]},E^{[n]})\quad,\quad q(\alpha)=\alpha\circ \alpha\,,
\]
as $q(\alpha)$ is the obstruction to extending the first order deformation corresponding to $\alpha$ over $\IC[t]/t^3$; see \cite[Exa.\ 2.16]{Manetti-deftheo}, \cite[Sect.\ 4.2]{Wandel2}. 
Since $\pi_*^{\sym_n}$ is a functor, hence compatible with the Yoneda product, we have a commutative diagram
\begin{align}\label{eq:cdYoneda}
\begin{CD}
\Ext_{\sym_n}^1(\CC^0_E,\CC^0_E)
@>{q}>>
\Ext_{\sym_n}^2(\CC^0_E,\CC^0_E)
\\
@VVV
@VVV
\\
\Ext^1(E^{[n]},E^{[n]})
@>{q}>>
\Ext^2(E^{[n]},E^{[n]})\,.
\end{CD} 
\end{align}
where the horizontal maps are the Yoneda squares, and the vertical maps are the edge morphisms of the spectral sequence of
\autoref{prop:spectralseq}. Let us first consider the case $\deg(E)> 0$, where \autoref{lem:extterms} gives isomorphisms
\begin{align}
 \Ext^1(\CC^0_E,\CC^0_E)&\cong \Ext^1(E,E) \oplus \bigl(\Hom(E,E)\otimes \Ho^1(\reg))\oplus \bigl(\Ho^1(E^\vee)\otimes \Ho^0(E)\bigr)\,,\notag\\
\Ext^2(\CC^0_E,\CC^0_E)&\cong \begin{matrix}\bigl(\Ext^1(E,E)\otimes \Ho^1(\reg_X)\bigr) \oplus \bigl(\Hom(E,E)\otimes\wedge^2\Ho^1(\reg_X)\bigr)\\ \oplus\bigl(\Ho^1(E^\vee)\otimes \Ho^0(E)\otimes \Ho^1(\reg_X)\bigr)\oplus\bigl(\Ho^1(E^\vee)\otimes \Ho^1(E)\bigr)\,.\end{matrix}\label{eq:Ext2dec}
\end{align}
Under these isomorphisms, the component $\widehat q\colon \Ho^1(E^\vee)\otimes \Ho^0(E)\to \Ho^1(E^\vee)\otimes \Ho^0(E)\otimes \Ho^1(\reg_X)$ of $q$ is (up to a non-zero scalar multiple) given by
\begin{align}\label{eq:qformula}
 \widehat q(\theta\otimes s)= \theta\otimes s\otimes(\theta\circ s)\,.
\end{align}
The formula \eqref{eq:qformula} can be verified by a computation similar to the one in the proof of \autoref{lem:ddescription}. Alternatively, the reader may consult \cite[Prop.\ 4.2]{KruExt}. There, a general formula for the Yoneda product on tautological bundles on Hilbert schemes of points on \textit{surfaces} is given. Furthermore, in the surface case, there is an isomorphism $\Ext^*_{\sym_n}(\CC^0_E, \CC^0_F)\cong \Ext^*(E^{[n]}, F^{[n]})$ where the definition of $\CC^0_E$ as an $\sym_n$-equivariant bundle is the same as in the curve case: $\CC^0_E=\Ind_{\sym_{n-1}}^{\sym_n}\pr_1^*E$. Hence, \cite[Prop.\ 4.2]{KruExt} also provides a formula for the Yoneda product on $\CC^0_E$ and this formula has to be the same independently of the dimension of $X$.

By \autoref{lem:Etermsd>0}(iii), the edge morphism $\Ext_{\sym_n}^2(\CC^0_E,\CC^0_E)\cong\IE^1_{0,2}\to \IE_2\cong \Ext^2(E^{[n]},E^{[n]})$ is identified with the cokernel of $d\colon \IE^1_{-1,2}\to \IE^1_{0,2}$. Hence, considering \eqref{eq:cdYoneda}, it is enough to proof that \[                                                                                                                                                                                                                                                                                                                                                                                                                                                                                                                                                         \im(\widehat q)\not\subset \im(\tilde d)\quad,\quad \tilde d\colon \IE^1_{-1,2}\xrightarrow d \IE^1_{0,2}\xrightarrow{c}  \Ho^1(E^\vee)\otimes \Ho^0(E)\otimes \Ho^1(\reg_X) \]
where $c$ is the projection to the third direct summand of the decomposition \eqref{eq:Ext2dec}. By the description of the dual of $d\colon \IE^1_{-1,2}\to \IE^1_{0,2}$ in \autoref{lem:dd>0}, we have $\im(\tilde d)=\im(\widehat d)$, where $\widehat d=\tilde d_{\mid \Ho^0(E)\otimes \Ext^1(E,\omega^{-1})}$. Note that $\im(\widehat d)\subset   \Ho^1(E^\vee)\otimes \Ho^0(E)\otimes \Ho^1(\reg_X)$ is a linear subspace. Hence, it suffices to prove
\begin{align}\label{eq:goal}
\langle\im(\widehat q)\rangle\not\subset \im(\widehat d) 
\end{align}
where $\langle\im(\widehat q)\rangle$ is the linear hull. Again by \autoref{lem:dd>0}, we know that $\widehat d=\id_{\Ho^0(E)}\otimes u$ for some\footnote{Concretely, under the isomorphisms $\Ext^1(E,\omega^{-1})\cong \Ext^1(E\otimes\omega,\reg)$ and $\Ho^1(E^\vee)\otimes \Ho^1(\reg_X)\cong \Hom_{\IC}\bigl(\Ho^0(\omega),\Ho^1(E^\vee)\bigr)$, we have $u(\phi)(t)=t\circ(\id_E\otimes t)$, but this is not important for our argument.} linear map $u\colon \Ext^1(E,\omega^{-1})\to \Ho^1(E^\vee)\otimes \Ho^1(\reg_X)$. The equality $\widehat d=\id_{\Ho^0(E)}\otimes u$ is understood by permuting the first two tensor factors in the target space. By Serre duality and Riemann--Roch, 
\begin{align}\label{eq:dimu}
 \dim(\im u)\le\ext^1(E,\omega^{-1})=\ho^0(E\otimes \omega^2)=d+3r(g-1)\,.
\end{align}
We now choose $\mu\gg 0$ (hence $d\gg r$) such that
\begin{align}\label{eq:ineqd}
d+3r(g-1)<\bigl(d+(r-1)(g-1)\bigr)g
\end{align}
We fix some $0\neq s\in \Ho^0(E)$. Since $s\colon \reg_X\to E$ is injective, the induced map 
\[
s^*\colon\Ho^1(E^\vee)\to \Ho^1(\reg_X) \quad,\quad \theta\mapsto \theta\circ s  
\]
is surjective. Now, choose a basis \[\theta_1,\dots, \theta_g,\theta'_{g+1},\dots,\theta'_{d+r(g-1)}\] 
of $\Ho^1(E^\vee)$ with $\ker(s^*)=\langle \theta'_{g+1},\dots,\theta'_{d+r(g-1)}\rangle$, and set $t_i:=s^*\theta_i=\theta_i\circ s$. Then, by \eqref{eq:qformula}, we have 
\[
 \widehat q(\theta_i\otimes s)=\theta_i\otimes s\otimes t_i\quad,\quad \widehat q((\theta_i+\theta'_j)\otimes s)=(\theta_i+\theta_j')\otimes s\otimes t_i\,. 
\]
Since all of the above vectors are linearly independent for varying $i=1,\dots,g$ and $j=g+1,\dots ,d+r(g-1)$, we have 
\[
\dim\Bigl(\langle \im(\widehat q)\rangle\cap\bigl(\Ho^1(\theta)\otimes \langle s\rangle \otimes \Ho^1(\reg)\bigr)\Bigr)\ge \bigl(d+(r-1)(g-1)\bigr)g\,. 
\]
On the other hand, $\im(\widehat d)\cap \bigl(\Ho^1(\theta)\otimes \langle s\rangle \otimes \Ho^1(\reg)\bigr)= \langle s\rangle\otimes \im(u)$ where the equality has to be understood by switching the first two tensor factors. 
Hence, \eqref{eq:dimu} and  \eqref{eq:ineqd} imply \eqref{eq:goal}.

As the $\deg E<0$ case is similar, we will keep the explanation a bit shorter. In this case, we consider the component
$\check q\colon \Ho^0(E^\vee)\otimes \Ho^1(E)\to \Ho^0(E^\vee)\otimes \Ho^1(E)\otimes \Ho^1(\reg_X)$ of the Yoneda square $q\colon \Ext^1(\CC^0_E,\CC_E^0)\to \Ext^2(\CC^0_E,\CC_E^0)$. We still have the formula
\begin{equation}\label{eq:qformula2}
 \check q(\theta\otimes s)=\theta\otimes s\otimes (\theta\circ s)\,,
\end{equation}
only that the classes $\theta$ and $s$ are now of different cohomological degree than in \eqref{eq:qformula}. By \autoref{lem:ddescription<0}, the image of the differential $d\colon \IE^1_{-1,2}\to \IE^1_{0,2}$ followed by the projection to  
$\Ho^0(E^\vee)\otimes \Ho^1(E)\otimes \Ho^1(\reg_X)$ equals the image of the component
\[
 \check d\colon \Ho^1(E)\otimes \Hom(E,\omega^{-1}) \to \Ho^0(E^\vee)\otimes \Ho^1(E)\otimes \Ho^1(\reg_X)
\]
Because of the vanishing in \autoref{lem:Einfty<0}(iii), it is again sufficient to prove that 
 $\langle \im(\check q)\rangle \not\subset \im(\check d)$.  Again by \autoref{lem:ddescription<0}, we have $\check d=\id_{\Ho^1(E)}\otimes v$ for some linear map \[v\colon \Hom(E,\omega^{-1})\to \Ho^0(E^\vee)\otimes \Ho^1(\reg_X)\,.\] This means that, for $\theta\in \Ho^0(E^\vee)$, $s\in \Ho^1(E)$, and $t\in \Ho^1(\reg)$, we have
 \begin{equation}\label{eq:implication}
 \theta\otimes s\otimes t \in \im(\check d)\quad\implies \quad \theta\otimes t\in \im(v)\,. 
 \end{equation}
Let $0\neq \theta\colon E\to \reg$. We claim that $\theta_*\colon \Ho^1(E)\to \Ho^1(\reg)$, $s\mapsto \theta\circ s$ is surjective. Write $K:=\ker(\theta)$. Since $\theta$ factors as $E\to E/K\xrightarrow{\overline \theta}\reg$, the push-forward $\theta_*$ on $\Ho^1$ factors as 
\begin{align}\label{eq:thetafact}
\Ho^1(E)\to \Ho^1(E/K)\xrightarrow{\overline\theta_*}\Ho^1(\reg)  
\end{align}
The first map $\Ho^1(E)\to \Ho^1(E/K)$ is surjective since $\Ho^2(K)=0$. Since $\coker(\overline \theta)$ has zero-dimensional support, we have $\Ho^1(\coker(\overline \theta))=0$ and the second map $\overline\theta_*$ is surjective too. 

We now assume for a contradiction that $\langle \im(\check q)\rangle \subset \im(\check d)$.
The surjectivity of $\theta_*$ (for every $\theta\neq 0$), together with \eqref{eq:qformula2} and \eqref{eq:implication} imply that every pure tensor $\theta\otimes t\in \Ho^0(E^\vee)\otimes \Ho^1(\reg)$ is contained in the linear subspace $\im(v)$. It follows that $v\colon \Hom(E,\omega^{-1})\to \Ho^0(E^\vee)\otimes \Ho^1(\reg_X)$ is surjective. But this cannot hold as $\hom(E,\omega^{-1})<\ho^0(E^\vee)\otimes \ho^1(\reg)$.
\end{proof}

\section{The genus 1 case}

In this section, let $X$ be curve of genus $g=1$.

\begin{lemma}\label{lem:ellipticvanish}
 Let $\reg_X\neq E\in \VB(X)$ be a simple vector bundle. Then 
\[\Ho^0(E)\otimes\Ho^0(E^\vee)=0=\Ho^1(E)\otimes\Ho^1(E^\vee)\,.\]  
\end{lemma}
\begin{proof}
If $\Ho^0(E)\neq 0\neq \Ho^0(E^{\vee})$, we would have a non-trivial endomorphism of $E$. The statement for $\Ho^1$ follows by Serre duality.
\end{proof}

Let us describe the leftmost differentials of the spectral sequence of \autoref{prop:spectralseq} for $E=F$ a simple bundle and $g=1$.

\begin{lemma}\label{lem:ddg1}
Let $\reg_X\neq E\in \VB(X)$ be a simple vector bundle. Then, for every $p\ge 1$ the dual of the differential $\IE^1_{-p,p}\to \IE^1_{-p+1,p}$ is isomorphic to 
\[
 (\cup\,,\,\id\,,\, \cup\,,\,\cup)\colon \begin{matrix} \bigl(\Hom(E,E)\otimes \Ho^1(\reg)\bigr)\oplus\Ext^1(E,E)\\\oplus \bigl(\Ho^0(E)\otimes \Ho^1(E^\vee)\bigr)\oplus \bigl(\Ho^1(E)\otimes \Ho^0(E^\vee)\bigr)\end{matrix}   \to \Ext^1(E,E)\,.
\]
In particular, the first two components are isomorphisms between one-dimensional vector spaces.  
\end{lemma}

\begin{proof}
This is found to be a special case of \autoref{lem:ddescription}, analogously to the proof of \autoref{lem:dd>0}.
\autoref{lem:ellipticvanish} is used for the vanishing of certain direct summands of $(\IE^1_{-p,p})^\vee$ and $(\IE^1_{-p+1,p})^\vee$. Also note that we again hide the one-dimensional vector space $\Ho^1(\omega)$ and its symmetric powers from the notation. 
\end{proof}

\begin{cor}\label{cor:Evanish}
For $\reg_X\neq E\in \VB(X)$ simple, we have $\IE^2_{-p,p}=0$ for all $p=1,\dots, n-1$.
\end{cor}

\begin{proof}
By definition $\IE^2_{-p,p}$ is the kernel of the differential $\IE^1_{-p,p}\to \IE^1_{-p+1,p}$. By \autoref{lem:ddg1}, this differential is injective as its dual is surjective. 
\end{proof}

\begin{theorem}\label{thm:simpleg1}
 Let $\reg_X\neq E\in \VB(X)$ be simple. Then $E^{[n]}$ is again simple for all $n\in \IN$. 
\end{theorem}

\begin{proof}
By \autoref{cor:Evanish}, we have $\Hom(E^{[n]},E^{[n]})\cong \IE_0\cong \IE^{\infty}_{0,0}\cong \IE^{1}_{0,0}$ which by \autoref{lem:dual} and \autoref{lem:ellipticvanish} is
$\IE^1_{0,0}\cong \Hom(E,E)$.
\end{proof}

\begin{remark}
 Independently of the genus of $X$, the tautological bundle $\reg_X^{[n]}$ is never simple as it has the structure sheaf $\reg_{X^{(n)}}$ as a direct summand. This is because $\pr_{X^{(n)}}\colon \Xi\to X^{(n)}$ is flat and finite, hence the pull-back of regular functions $\reg_{X^{(n)}}\to \pr_{X^{(n)}*}\reg_\Xi\cong \reg_X^{[n]}$ has $\frac 1n$ times the trace map as a left-inverse.
\end{remark}

\begin{remark}
 Note that simplicity is not preserved for genus $g=0$. For example, $\reg_{\IP}(1)^{[2]}\cong \reg_{\IP^2}\oplus\reg_{\IP^2}$; see \cite[Sect.\ 3]{Nagaraj--sym} or \cite[Sect. 6.1]{Krug--remarksMcKay}.
\end{remark}

\begin{prop}\label{prop:Extinjg1}
 Let $E\neq \reg_X\in \VB(X)$ be a simple vector bundle. Then 
\[(\_)^{[n]}\colon \Ext^1(E,E)\to \Ext^1(E^{[n]},E^{[n]})\] is injective for all $n\in \IN$.
\end{prop}

\begin{proof}
 By \autoref{rem:remsuffice} and \autoref{cor:Evanish}, it is enough to prove that, for every $\psi\in \Ext^1(E,E)$, the class $\CC(\psi)$ is not contained in the image of the differential $\IE^1_{-1,1}\to \IE^1_{0,1}$. Under the isomorphism
 \begin{equation}\label{eq:ext1decomp}
  \Ext^1(\CC(E),\CC(E))\cong \begin{matrix}\Ext^1(E,E)\oplus\bigl(\Hom(E,E)\otimes \Ho^1(\reg)\bigr)\\\oplus \bigl(\Ho^1(E^\vee)\otimes \Ho^0(E)\bigr)\oplus \bigl(\Ho^0(E^\vee)\otimes \Ho^1(E)\bigr)\end{matrix}
 \end{equation}
we have $\CC(\psi)=(\psi,0,0,0)$. But, by \autoref{lem:ddg1}, the differential maps $\IE^1_{-1,1}\cong \Hom(E,E)$ isomorphically to \textit{both} of the first two summands of \eqref{eq:ext1decomp}. Hence, an element of the form $(\psi,0,0,0)$, with a zero in the second component but a non-zero first component, can not be in the image of the differential.   
\end{proof}

As explained in \autoref{rem:mainsuffice}, this finishes the proof of \autoref{thm:mainclosed}.

\bibliographystyle{alpha}
\addcontentsline{toc}{chapter}{References}
\bibliography{references}

\end{document}